\documentclass[11pt]{article}
\usepackage{color}

\usepackage{amssymb}   
\usepackage{amsthm}    
\usepackage{amsmath}   
\usepackage{stmaryrd}  
\usepackage{titletoc}  
\usepackage{mathrsfs}  
\usepackage{graphicx}

\newlength{\defbaselineskip}
\newcommand{\setlinespacing}[1]%
           {\setlength{\baselineskip}{#1 \defbaselineskip}}

\theoremstyle{plain}
\newtheorem{thm}{Theorem}[section]
\newtheorem{cor}[thm]{Corollary}
\newtheorem{lem}[thm]{Lemma}
\newtheorem{prop}[thm]{Proposition}

\theoremstyle{definition}
\newtheorem{defn}{Definition}[section]

\newtheorem{rmk}{Remark}[section]

\newcommand{\eps}{\varepsilon}

\DeclareMathOperator*{\esssup}{esssup}
\DeclareMathOperator*{\essinf}{essinf}

\newcommand{\cL}{\mathcal{L}}

\newcommand{\cB}{\mathcal{B}}
\newcommand{\cA}{\mathcal{A}}
\newcommand{\cS}{\mathcal{S}}
\newcommand{\cE}{\mathcal{E}}

\newcommand{\cU}{\mathcal{U}}

\newcommand{\cD}{\mathcal{D}}

\newcommand{\bH}{\mathbb{H}}
\newcommand{\bP}{\mathbb{P}}
\newcommand{\bR}{\mathbb{R}}
\newcommand{\bN}{\mathbb{N}}

\newcommand{\sF}{\mathscr{F}}
\newcommand{\sP}{\mathscr{P}}

\makeatletter\@addtoreset{equation}{section} \makeatother
 \allowdisplaybreaks
\begin{document}

\title{Weak Solution for Fully Nonlinear  Stochastic Hamilton-Jacobi-Bellman Equations
}

\author{Jinniao Qiu\footnotemark[1]  }
\footnotetext[1]{Department of Mathematics, Humboldt-Universit\"at zu Berlin, Unter den Linden 6, 10099 Berlin, Germany. \textit{E-mail}: \texttt{qiujinn@gmail.com}. Financial support from the chair Applied Financial Mathematics  is gratefully acknowledged. The author would like to thank Professor Shanjian Tang for the kind encouragement and helpful discussions.}

%
%

\maketitle

\begin{abstract}
This paper is concerned with the stochastic Hamilton-Jacobi-Bellman equation with controlled leading coefficients, which is a type of fully nonlinear backward stochastic partial differential equation (BSPDE for short). In order to formulate the  weak solution for such kind of BSPDEs, a class of regular random parabolic potentials are introduced in the backward stochastic framework. The existence and uniqueness of  weak solution is proved, and for the \textit{partially} non-Markovian case, we obtain the associated gradient estimate. As a byproduct, the existence and uniqueness of solution for a class of \textit{degenerate} reflected BSPDEs is discussed as well.
\end{abstract}



\section{Introduction}
Let $(\Omega,\sF,\{\sF_t\}_{t\geq0},\bP)$ be a complete filtered probability space on which is defined an $m$-dimensional Wiener process $W=\{W_t:t\in[0,\infty)\}$ such that $\{\sF_t\}_{t\geq0}$ is the natural filtration generated by $W$ and augmented by all the
$\bP$-null sets in $\sF$. We denote by $\sP$ the $\sigma$-algebra of the predictable sets on $\Omega\times[0,T]$ associated with $\{\sF_t\}_{t\geq0}$. Consider the following fully nonlinear BSPDE:
\begin{equation}\label{SHJB}
  \left\{\begin{array}{l}
  \begin{split}
  -du(t,x)=\,&\displaystyle 
  \essinf_{\sigma\in U} \bigg\{\text{tr}\left(\frac{1}{2}\sigma \sigma' D^2 u+\sigma D\psi\right)(t,x)
        +f(t,x,\sigma)
                \bigg\} \,dt-\psi(t,x)\, dW_{t},\\ &\displaystyle
            \quad
                     (t,x)\in Q:=[0,T]\times \bR^d;\\
    u(T,x)=\, &G(x), \quad x\in\bR^d.
    \end{split}
  \end{array}\right.
\end{equation}
Here and in the following $U$ is a nonempty bounded subset of $R^{d\times m}$,  $T\in(0,\infty)$  is a fixed deterministic terminal time, and $D$ and $D^2$ {denote} respectively the gradient operator and  the second-order differential operator. A solution of BSPDE (\ref{SHJB}) is a random couple $(u,\psi)$ defined on $\Omega\times[0,T]\times\bR^d$ such that   (\ref{SHJB}) holds in the sense of Definition \ref{def-weak-sltn-HJB} below.

The fully nonlinear BSPDE like \eqref{SHJB} is also called stochastic Hamilton-Jacobi-Bellman (HJB) equation, which was first introduced by Peng \cite{Peng_92} to characterize the value function for the stochastic optimal control problem of non-Markovian type. For the utility maximization with habit formation, a specific fully nonlinear stochastic HJB equation was formulated by Englezos and Karatzas \cite{EnglezosKaratzas09} and the value function was verified to be its classical solution. The study of linear BSPDEs  dates back to  about thirty years ago (see Bensoussan
\cite{Bensousan_83} and Pardoux \cite{Pardoux1979}). They arise in many applications of probability theory and stochastic processes, for
instance in the nonlinear filtering and stochastic control theory for processes with incomplete information, as an adjoint equation of the Duncan-Mortensen-Zakai filtration equation (for instance, see \cite{Bensousan_83,Hu_Ma_Yong02,Hu_Peng_91,Tang_98,Zhou_93}). The representation relationship between forward-backward stochastic differential equations and BSPDEs yields the stochastic Feynman-Kac formula (see \cite{Hu_Ma_Yong02,ma1999linear,QiuTangYou-SPA-2011}). In addition, as the obstacle problems of BSPDEs, the reflected BSPDE arises as the HJB equation for the optimal stopping problems (see \cite{ChangPangYong-2009,Oksend-Sulem-Zhang-2011,QiuWei-RBSPDE-2013,Tang-Yang-2011}).

The linear and semilinear BSPDEs have been extensively studied, we refer to \cite{DuQiuTang10,DuTang2010,DuTangZhang-2013,Hu_Ma_Yong02,Hu_Peng_91,ma2012non,ma1999linear,Tang-Wei-2013} among many others. For the weak solutions and associated local behavior analysis for general quasi-linear BSPDEs, see \cite{QiuTangBDSDES2010,QiuTangMPBSPDE11}, and we refer to \cite{GraeweHorstQui13,Horst-Qiu-Zhang-14} for BSPDEs with singular terminal conditions. However, for the fully nonlinear case, there are few results on the stochastic HJB equations, even for the simplified cases like \eqref{SHJB}. The existence and uniqueness of solution for stochastic HJB equations with controlled leading coefficients is still an open problem, which is claimed
in Peng's plenary lecture of ICM 2010 (see \cite{peng1999open,peng2011backward}).


Put
\begin{align*}
\mu^{\sigma}(dt,dx)=\bigg( &\text{tr}\left(\frac{1}{2}\sigma \sigma' D^2 u+\sigma D\psi\right)(t,x)
        +f(t,x,\sigma)\\
        &-\essinf_{\bar\sigma\in U} \bigg\{\text{tr}\left(\frac{1}{2}\bar\sigma \bar\sigma' D^2 u+\bar\sigma D\psi\right)(t,x)
        +f(t,x,\bar\sigma)
                \bigg\}\bigg)dtdx.
\end{align*}
Then BSPDE \eqref{SHJB} reads
\begin{equation}\label{SHJB-sigm}
  \left\{\begin{array}{l}
  \begin{split}
  -du(t,x)+\mu^{\sigma}(dt,x)=\,&\displaystyle 
  \bigg\{\text{tr}\left(\frac{1}{2}\sigma \sigma' D^2 u+\sigma D\psi\right)(t,x)
        +f(t,x,\sigma)
                \bigg\} \,dt-\psi(t,x)\, dW_{t},\\ &\displaystyle
            \quad
                     (t,x)\in Q;\\
    u(T,x)=\, &G(x), \quad x\in\bR^d.
    \end{split}
  \end{array}\right.
\end{equation}
Heuristically, $\mu^{\sigma}(dt,dx)$ can be seen as a random measure and if the family of triples $(u,\psi,\mu^{\sigma})$ satisfying BSPDE \eqref{SHJB-sigm} such that the infimum of family $\{\mu^{\sigma}\}$ indexed by $\sigma$ vanishes, then one can conjecture that  $(u,\psi)$ should be a weak solution for BSPDE \eqref{SHJB} in some sense. For the Markovian case where the coefficients $f$ and $G$ are deterministic functions, BSPDE \eqref{SHJB} becomes a classical \textit{deterministic}  HJB equation and a similar idea was conjectured by Lions \cite{lions-1983-HJB}, but to the best of our knowledge, the only existing partial result along this line owes to Coron and Lions \cite{coron-lions-1986-remark} for the one-dimensional elliptic case.

To formulate the weak solution, we characterize first the random measure $\mu^{\sigma}$. Inspired by the recent work on BSPDEs with random measures by Qiu and Wei \cite{QiuWei-RBSPDE-2013}, and incorporating the degenerateness of BSPDE \eqref{SHJB-sigm}, we introduce a class of regular parabolic potentials  in the backward stochastic framework and identify the measure $\mu^{\sigma}$ as the corresponding regular random Radon measure. Such regular potential is  equivalently described via backward stochastic differential equation (BSDE) and in a variational way respectively. Furthermore, a monotonic convergence theorem is proved, the regular potential is further characterized as its own Snell envelope, and as a byproduct, the existence and uniqueness of solution for a class of \textit{degenerate} reflected BSPDEs is obtained.   These results are presented in Section 3.

Basing on the generalized potential theory, we give the definition and prove the existence and uniqueness of weak solution for BSPDE \eqref{SHJB} (see Theorem \ref{thm-EU-main} for the main result). For the partially non-Markovian case where the randomness of the coefficients $f$ and $G$ is from the filtration $\{\tilde\sF_t\}_{t\in[0,T]}$ generated by $\tilde{W}:=(W^1,\dots,W^{m_0})$ ($m_0<m$) that is part of the Wiener process $W$, the solution $(u,\psi)$ is verified to be only adapted to $\{\tilde\sF_t\}$ and the gradient estimates are obtained. The reader can find such main results in Section 4 and a short comment on generalizations in Section 5. 

In addition, we set notations and list the standing assumptions in Section 2, and in the appendix, we recall the generalized  It\^o-Wentzell formula by  Krylov \cite{Krylov_09} and the existence and uniqueness of solution for a class of Banach space-valued BSDEs, from which the existence and uniqueness of solution for a class of degenerate BSPDEs is derived.



\section{Preliminaries}

Denote by $|\cdot|$ the norm in  Euclidean spaces.
For each $l\in \mathbb{N}^+$ and domain $\Pi\subset \bR^l$, denote by $C_c^{\infty}(\Pi)$ the space of infinitely differentiable functions with compact supports in $\Pi$. We write $C_c^{\infty}:=C_c^{\infty}(\bR^l)$ when there is no confusion on the dimension. In this work, we shall use
$\mathcal{D}_T:=C_c^{\infty}(\bR)\otimes C_c^{\infty}(\bR^d)$ as the space of test functions in the definition of weak solutions for BSPDEs. The Lebesgue measure in $\bR^d$ will be denoted by $dx$. $L^2(\bR^d)$ ($L^2$ for short) is the usual Lebesgue integrable space with scalar product and norm defined
$$
\langle \phi,\,\psi\rangle=\int_{\bR^d}\phi(x)\psi(x)dx,\quad \|\phi\|=\langle\phi,\,\phi\rangle^{1/2},\,\,\forall
\phi,\psi\in L^2.
$$
In addition, for each $(n,p)\in\bR\times (1,\infty)$ we define  the $n$-th order Bessel potential space $(H^n_p,\|\cdot\|_{n,p})$ as usual (see Appendix \ref{sec:Banch-BSDE}). For convenience, we shall also use $\langle \cdot,\,\cdot\rangle$ to denote the duality between $(H^n_p)^k$ and $(H^{-n}_q)^k$ ($k\in\bN^+,\,n\in\bR,\,\frac{1}{p}+\frac{1}{q}=1$) as well as that between the Schwartz function space $\mathscr{D}$ and  $C_c^{\infty}$.

Let $V$ be a Banach space equipped with norm $\|\cdot\|_V$. For $p\in[1,\infty]$, $\cS ^p (V)$ is the set of all the $V$-valued,
 $(\sF_t)$-adapted and continuous processes $\{X_{t}\}_{t\in [0,T]}$ such
 that
{\small $$\|X\|_{\cS ^p(V)}:= \left\|\sup_{t\in [0,T]} \|X_t\|_V\right\|_{L^p(\Omega,\sF,\bP)}< \infty.$$
}
 Denote by $\mathcal{L}^p(V)$ the totality of all  the $V$-valued,
 $(\sF_t)$-adapted processes $\{X_{t}\}_{t\in [0,T]}$ such
 that
 {\small
 $$
 \|X\|_{\mathcal{L}^p(V)}:=\left\| \bigg(\int_0^T \|X_t\|_V^2\,dt\bigg)^{1/2} \right\|_{L^p(\Omega,\sF,\bP)}< \infty.
 $$
 }
Obviously, $(\cS^p(V),\,\|\cdot\|_{\cS^p(V)})$ and $(\mathcal{L}^p(V),\|\cdot\|_{\mathcal{L}^p(V)})$
are Banach spaces.

By convention, we treat elements of spaces like $\cS^2(H^n_2)$ and $\cL^2(H^n_2)$ as functions rather than distributions or classes of equivalent functions, and if a function of such class admits a version with better properties, we always denote this version by itself. For example, if $u\in \cL^2(H^n_2)$ and $u$ admits a version lying in $\mathcal{S}^2(H^n_2)$, we always adopt the modification $u\in \cL^2(H^n_2)\cap \cS^2(H^n_2)$.
\medskip

Consider BSPDE~(\ref{SHJB}). We define the following assumption.

\bigskip\medskip
   $({\mathcal A} 1)$ \it $G\in L^{2}(\Omega,\sF_T;L^2)$ is nonnegative and the   random function
$
  f:~\Omega\times[0,T]\times\bR^d\times U\rightarrow[0,\infty)
$
is $\sP\otimes\cB(\bR^d)\otimes\cB(U)$-measurable. There exist $\alpha\in(0,1]$ and $L>0$ and some $g\in \cL^2(L^2)$   such that for all $x_1,x_2\in \bR^d$, $v\in U$   and $(\omega,t)\in \Omega\times[0,T]$,
   $ f(\omega,t,x_1,v)
       \leq\,
       g(\omega,t,x_1)$
       and
   \begin{equation*}
     \begin{split}
       &|f(\omega,t,x_1,v)-f(\omega,t,x_2,v)|+|G(\omega,x_1)-G(\omega,x_2)|
       \leq\, L|x_1-x_2|^{\alpha}.
     \end{split}
   \end{equation*}\rm

Note that in this work, $f(t,x,v)$ does not necessarily depend continuously on the \textit{control} $v$. In view of BSPDE \eqref{SHJB}, we also note that the nonnegativity of $G$ and $f$ is assumed for simplicity and that it can be replaced equivalently by the lower-boundedness.

%
%
%

\section{$\sigma$-quasi-continuity and regular $\sigma$-potential}
Throughout this work, denote by $\mathcal{U}$ the set of all the $U$-valued and ${\sF}_t$-adapted processes and for each $\sigma\in \cU$,
\begin{align}
X^{\sigma}_t:=\int_0^t\sigma_s\,dW_s,\quad t\in[0,T].
\end{align}

In this section, we fix some $\sigma\in\cU$. Obviously, one has
\begin{align}
E\int_0^T\|h(t,\cdot+X_t^{\sigma})\|^2dt=\|h\|^2_{\cL^2(L^2)} \leq T\int_{\bR^d}E\sup_{t\in[0,T]}|h(t,x+X^{\sigma}_t)|^2\,dx.\label{eq-equiv-norm}
\end{align}
\subsection{$\sigma$-quasi-continuity }
\begin{defn}
  Random function $u:\Omega\times [0,T]\times \bR^d\rightarrow \bar{\bR}$ is said to be $\sigma$-quasi-continuous provided that for each $\eps >0$, there exists a predictable random set $D^{\eps}\subset \Omega\times [0,T]\times \bR^d$ such that $\mathbb{P}$-a.s. the section $D^{\eps}_{\omega}$ is open and $u(\omega,\cdot,\cdot)$ is continuous on its complement $(D^{\eps}_{\omega})^{c}$ and
  $$
  {\mathbb{P}}\otimes{dx}
  \left((\omega,x)| \exists t\in [0,T]\,\,\textrm{s.t.}\,\,(\omega,t,x+X^{\sigma}_t(\omega))\in D^{\eps} \right)\leq \eps.
  $$
\end{defn}

  If $u$ is $\sigma$-quasi-continuous, we can check that the process $\{u(t,x+X^{\sigma}_t)\}_{t\in[0,T]}$ has continuous trajectories, ${\bP}\otimes{dx}$-a.e.   In order to verify the $\sigma$-quasi-continuity of some random function, we always use the following lemma on the closeness.

\begin{lem}\label{lem-S-quasi-contin}
  Let $\{u_n\}_{n\in\bN^+}$ be a sequence of $\sigma$-quasi-continuous processes.
  Assume that  there exists random function $u$ such that for some $p\in(0,\infty)$
  \[
  E\int_{\bR^d}\sup_{t\in[0,T]}|u(t,x+X^{\sigma}_{t})|^p\,dx<\infty
  \quad \text{and}\quad
  \lim_{n\rightarrow\infty}E\int_{\bR^d}\sup_{t\in[0,T]}|u(t,x+X^{\sigma}_{t})-u_n(t,x+X^{\sigma}_{t})|^p\,dx=0.\]
  Then $u$ is $\sigma$-quasi-continuous.
\end{lem}

\begin{proof}
For each $\delta\in(0,1)$ and $n\in\bN^+$, since $u_n$ is $\sigma$-quasi-continuous, there exists $D^{\delta_n}$ whose section $D^{\delta_n}_{\omega}$ is open  such that $u_n(\omega,\cdot,\cdot)$ is continuous on its complement $(D^{\delta_n}_{\omega})^c$ and
$$
{\bP}\otimes{dx}\left(
(\omega,x)| \exists t\in [0,T]\,\,\textrm{s.t.}\,\,(\omega,t,x+X_t^{\sigma}(\omega))\in D^{\delta_n}
\right)\leq \frac{\delta}{2^n}.
$$
 Put $D^{\delta}=\cup_{n} D^{\delta_n}$.
Choosing a subsequence if necessary, we assume
\[
E\int_{\bR^d}\sup_{t\in[0,T]}|u_{n+1}(t,x+X^{\sigma}_t)-u_n(t,x+X^{\sigma}_t)|^p\,dx<\frac{1}{2^n}.
\]
For each $\eps>0$ and $k,n\in\bN^+$, set $F^n=\{|u_n-u_{n-1}|>\eps\}$ and $D^k=\cup_{n\geq k}F^n$. Then
\begin{align*}
&\eps^{p}{\bP}\otimes{dx}\left((\omega,x)| \exists t\in [0,T]\,\,\textrm{s.t.}\,\,(\omega,t,x+X^{\sigma}_t(\omega))\in F^n   \right)
\\
\leq\,&
E\int_{\bR^d}\sup_{t\in[0,T]}|u_n(t,x+X^{\sigma}_t)-u_{n+1}(t,x+X^{\sigma}_t)|^p\,dx\leq \frac{1}{2^n}.
\end{align*}
Taking $\eps=\frac{1}{n^{2}}$, we get the continuity of $u(\omega,\cdot,\cdot)$ on the \textit{closed} complement of the section $D^k_{\omega}\cup D^{\delta}_{\omega}$ and
\begin{align*}
{\bP}\otimes {dx}\left((\omega,x)| \exists t\in [0,T]\,\,\textrm{s.t.}\,\,(\omega,t,x+X^{\sigma}_t(\omega))\in D^k \cup D^{\delta}  \right)
\leq
\delta+\sum_{n=k}^{\infty} \frac{n^{2p}}{2^n},
\end{align*}
which implies the $\sigma$-quasi-continuity of $u$.
\end{proof}
We are going to study the $\sigma$-quasi-continuity of weak solution for BSPDEs. Consider the following BSPDE:
\begin{equation}\label{SHJB-linr-sigm}
  \left\{\begin{array}{l}
  \begin{split}
  -du(t,x)=\,&\displaystyle 
  \left[
  \text{tr}\left(\frac{1}{2}\sigma \sigma' D^2 u+\sigma D\psi\right)(t,x)
        +f(t,x)
                \right] \,dt
           -\psi(t,x)\, dW_{t}, \quad
                     (t,x)\in Q;\\
    u(T,x)=\, &\Psi(x), \quad x\in\bR^d,
    \end{split}
  \end{array}\right.
\end{equation}
with $f\in \mathcal{L}^2(L^2)$, $\Psi\in L^2(\Omega,\sF_T;L^2)$. We recall that a weak solution of BSPDE \eqref{SHJB-linr-sigm} is a pair of processes $(u,\psi)\in \cS^2(L^2)\times \mathcal{L}^2((H^{-1}_2)^m)$ such that for each test function $\varphi\in C_c^{\infty}$ and any  $t\in[0,T]$, we have
\begin{equation*}
  \begin{split}
    &\langle u(t),\,\varphi \rangle \\
    =\,&
    \langle \Psi,\varphi \rangle+\!\!\int_t^T\!\!\left[\langle f(s),\,\varphi \rangle
    +\frac{1}{2}\langle  u(s),\, \text{tr}\left(\sigma\sigma'D^2\varphi\right)\rangle -\langle \sigma\psi(s),\, D\varphi \rangle \right]\,ds-\int_t^T\langle \varphi,\,\psi(s)\,dW_s\rangle,\ \text{a.s.},
  \end{split}
\end{equation*}
where, similar to \cite[Remark 2.1]{QiuTangMPBSPDE11}, the test function space $C_c^{\infty}(\bR^d)$ can be replaced by $\mathcal{D}_T$, i.e., for each test function $\varphi\in \mathcal{D}_T$ and any  $t\in[0,T]$,
\begin{equation*}
  \begin{split}
    &\langle u(t),\,\varphi(t) \rangle +\!\!\int_t^T\!\langle u(s),\,\partial_s \varphi(s)   \rangle  \,ds\\
    =\,&
    \langle \Psi,\varphi(T) \rangle+\!\!\int_t^T\!\left[\langle f,\,\varphi \rangle+\frac{1}{2}\langle  u,\, \text{tr}\left(\sigma\sigma'D^2\varphi\right)\rangle -\langle \sigma\psi,\, D\varphi \rangle \right](s)\,ds
    -\int_t^T\langle \varphi(s),\,\psi(s)\,dW_s\rangle ,\,\text{a.s.}
  \end{split}
\end{equation*}

Setting
$$
(\bar{u}(s,x),\bar{\psi}(s,x))=(u(s,X_s^{\sigma}+x),(\psi+ Du \sigma)(s,X_s^{\sigma}+x)),\quad (s,x)\in Q,
$$
from Theorem \ref {Ito-Wentzell} we conclude that $(u,\psi)$ is a solution of BSPDE \eqref{SHJB-linr-sigm} if and only if $(\bar{u},\bar{\psi})$ satisfies the following trivial one:
\begin{equation}\label{SHJB-linr-sigm-trvl}
  \left\{\begin{array}{l}
  \begin{split}
  -d\bar{u}(t,x)=\,&\displaystyle 
        f(t,x+X_t^{\sigma}) \,dt
           -\bar{\psi}(t,x)\, dW_{t}, \quad
                     (t,x)\in Q;\\
    \bar{u}(T,x)=\, &\Psi(x+X_T^{\sigma}), \quad x\in\bR^d.
    \end{split}
  \end{array}\right.
\end{equation}
Therefore, BSPDE \eqref{SHJB-linr-sigm} admits a unique solution by Proposition \ref{prop-banah-BSDE} and
{\small
\begin{align}
\|u\|^2_{\cS^2(L^2)}+\|\psi+Du\sigma\|^2_{\cL^2((L^2)^m)}+E\int_{\bR^d} \sup_{s\in[0,T]}|u(s,X_s^{\sigma}+x)|^2\,dx
\leq C\left\{
E\|\Psi\|^2+\|f\|^2_{\cL^2(L^2)}
\right\}.
\end{align}
}
Moreover, the process $\{u(t,x+X^{\sigma}_t)\}_{t\in[0,T]}$ has continuous trajectories, ${\bP}\otimes{dx}$-a.e.
%
%
%
In what follows, we denote the unique solution of BSPDE \eqref{SHJB-linr-sigm}   associated with $(\sigma,f,\Psi)$ by
$$(u,\psi):=\mathbb{S}(\sigma,f,\Psi). $$
 As an immediate consequence of  Theorem \ref {Ito-Wentzell}, we give without any proof the following
\begin{prop}\label{prop-BSPDE-Flow}
  Given $(u,\psi)=\mathbb{S}(\sigma,f,\Psi)$, one has the following stochastic representations, for $0\leq t\leq s\leq T$,
  \begin{equation*}
    \begin{split}
      &u(t,x+X^{\sigma}_t)+\!\int_t^s\!\!\!
      (\psi+ Du\sigma)(\tau,x+X^{\sigma}_{\tau})\,dW_{\tau}
      =\,u(s,x+X^{\sigma}_{s})+\!\!\int_t^s\!\!f(\tau,x+X^{\sigma}_{\tau})\,d\tau,\ {\bP}\otimes{dx}\text{-a.e.}
    \end{split}
  \end{equation*}
  with $\psi+Du\sigma\in\cL^2((L^2)^m)$.
\end{prop}

\begin{prop}\label{prop-quasi-cont-BSPDE}
  For $(u,\psi)=\mathbb{S}(\sigma,f,\Psi)$, $u$ is $\sigma$-quasi-continuous.
\end{prop}
\begin{proof}
  Given integer $k>\frac{d}{2}$, the $k$-th order Sobolev space $H^k_2$ is continuously embedded into the H\"{o}lder space $C^{\alpha}$ with $0<\alpha<1\wedge (k-\frac{d}{2})$. Let $\{(\Psi_n,f_n)\}_{n\in\bN^+}\subset L^2(\Omega,\sF_T;H^k_2)\times\cL^2(H^k_2)$ be a sequence converging to $(\Psi,f)$ in $ L^2(\Omega,\sF_T;L^2)\times \cL^2(L^2)$. Set $(u_n,\psi_n)=\mathbb{S}(\sigma,f_n,\Psi_n)$. By Corollary \ref{cor-degenerate-BSPDE}, $u_n\in \cS^2(H^{k}_2)$, i.e., $u_n$ is an $H^k_2$-valued continuous process and hence $u_n(\omega,t,x)$ is almost surely continuous in $(t,x)$. Moreover,
  \begin{align}
  E\int_{\bR^d}\sup_{t\in[0,T]}|(u_n-u)(t,x+X_t^{\sigma})|^2\,dx
  \leq\,& C \left(E\|\Psi-\Psi_n\|^2+\|f_n-f\|^2_{\cL^2(L^2)}\right)\nonumber\\
  &\longrightarrow 0,\quad \text{as }n\rightarrow\infty.
  \end{align}
Hence, by Lemma \ref{lem-S-quasi-contin}, $u$ is $\sigma$-quasi-continuous.
\end{proof}

\subsection{Regular $\sigma$-potential}

For each $s\geq 0$, we define operator $P^{\sigma}_s$ on $\cL^2(L^2)$ such that for each $u\in \cL^2(L^2)$,
  \begin{equation*}
  {P}^{\sigma}_s u(t_0,x):=
  \left\{\begin{array}{l}
  \begin{split}
  \tilde{u}(t_0,x),\ &\textrm{ if }s+t_0\leq T;\\
                           0,\ &\textrm{ otherwise,}
    \end{split}
  \end{array}\right.
 \end{equation*}
where $\tilde{u}$ together with some random field $\tilde{\psi}$ constitutes the weak solution to the following BSPDE
\begin{equation*}
  \left\{\begin{array}{l}
  \begin{split}
  -d\tilde{u}(t,x)=\,&\displaystyle 
  \text{tr}\left(\frac{1}{2}\sigma \sigma' D^2 \tilde{u}+\sigma D\tilde{\psi}\right)(t,x)
                 \,dt
           -\tilde{\psi}(t,x)\, dW_{t},
                     \,\,(t,x)\in [0,t_0+s]\times\bR^d;\\
    \tilde{u}(t_0+s,x)=\, &u(t_0+s,x), \quad x\in\bR^d.
    \end{split}
  \end{array}\right.
\end{equation*}
In view of Proposition \ref{prop-BSPDE-Flow}, we have another representation for $P^{\sigma}_s$, i.e.,
\begin{align}
{P}^{\sigma}_s u(t_0,x)=E_{{\sF}_{t_0}}[u(t_0+s,x+X^{\sigma}_{t_0+s}-X^{\sigma}_{t_0})],\quad 0\leq t_0\leq t_0+s\leq T.\label{P-sgm-exptn}
\end{align}
Therefore, for any $(\hat{u},\psi)=\mathbb{S}(\sigma,f,0)$, we have
\begin{align}
\hat{u}(t,x)=E_{{\sF}_t}\int_t^Tf(s,x+X^{\sigma}_{s}-X^{\sigma}_{t})\,ds
=\int_t^TP^{\sigma}_{s-t}f(t,x)\,ds,\quad (t,x)\in Q.\label{P-BSPDE-repn}
\end{align}
Moreover, it is obvious that $\|P^{\sigma}_su(t_0,\cdot)\|\leq \|u(t_0+s,\cdot)\|$ for any $(t_0,s)\in[0,T]\times [0,\infty)$. In view of representation \eqref{P-sgm-exptn}, we have further the following
\begin{lem}\label{lem-P-smgp}
 $({P}^{\sigma}_t)_{t\geq 0}$ is a strongly continuous one-parameter contraction semigroup on $\cL^2(L^2)$.
\end{lem}
\begin{proof}
It is sufficient to check that for each $u\in\cL^2(L^2)$, there holds
\begin{align}
\lim_{s\rightarrow 0}\left(
\int_0^{T-s}E\|P^{\sigma}_{s}u(t)-u(t)\|^2\,dt+\int_{T-s}^TE\|u(t)\|^2\,dt
\right)=0.\label{eq-lem-lem-P-smgp}
\end{align}
Notice that
\begin{align*}
&\int_0^{T-s}E\|P^{\sigma}_{s}u(t)-u(t)\|^2\,dt\\
&=\int_0^{T-s}E\left\|E_{\sF_t}\left[u(t+s,\cdot+X^{\sigma}_{t+s}-X^{\sigma}_t)-u(t,\cdot)\right]\right\|^2\,dt
\\
&\leq
\int_0^{T-s}E\left\|u(t+s,\cdot+X^{\sigma}_{t+s}-X^{\sigma}_t)-u(t,\cdot)\right\|^2\,dt.
\end{align*}
Fix some $t\in[0,T)$. For any $\zeta\in L^{\infty}(\Omega,\sF_{t+s};\bR)$ and $\phi\in \mathcal{D}_T$, one has
{\small
\begin{align*}
&E\|\zeta\phi(t+s,\cdot+X^{\sigma}_{t+s}-X^{\sigma}_t)-\zeta\phi(t,\cdot)\|^2
\\
&\leq C E\|\phi(t+s,\cdot+X^{\sigma}_{t+s}-X^{\sigma}_t)-\phi(t,\cdot)\|^2
\\
&\leq
C E\int_{\bR^d}\Big(
\int_0^{s}\!\!\left(\left|\partial_r\phi\right|+\left|D^2\phi\right|\right)^2(t+r,x+X^{\sigma}_{t+r}-X^{\sigma}_t)\,dr
+\int_0^{s}\left| D\phi(t+r,x+X^{\sigma}_{t+r}-X^{\sigma}_t)\right|^2\,dr\Big)\,dx
\\
&
\leq Cs\,\rightarrow 0,\text{ as }s\rightarrow 0,
\end{align*}
}
where $C$ is independent of $(s,t)$.
Then the standard density argument yields  \eqref{eq-lem-lem-P-smgp}.
\end{proof}
\begin{defn}
  $u\in \cS^2(L^2)$ is called a regular $\sigma$-potential, provided that $u$ is $\sigma$-quasi-continuous,
   $\lim_{t\rightarrow T}u(t,\cdot)=0$ in $L^2(\bR^d)$ a.s.,
  \begin{equation}\label{Eq0 Stocha-Potential}
    E\int_{\bR^d} \sup_{t\in[0,T]} |u(t,x+X^{\sigma}_t)|^2 \,dx< \infty,
  \end{equation}
  and
  \begin{equation}\label{Eq Stoch-Potential}
    P^{\sigma}_su(t)\leq u(t),\,\mathbb{P}\otimes dx\text{-a.e.}\,\,\,\forall\, (t,s)\in [0,T]\times(0,\infty).
  \end{equation}
\end{defn}

 In view of the above definition, it is obvious that each regular $\sigma$-potential is nonnegative.

\begin{thm}\label{thm stoch-Potential}
  Let $u\in\cS^2(L^2)$. Then $u$  is a regular $\sigma$-potential if and only if there exist random field $\psi\in \mathcal{L}^2((H^{-1}_2)^{m})$ and a continuous increasing process $K=\{K_t\}_{t\in[0,T]}$ such that $K_0=0$, $K_t$ is ${\sF}_t\otimes\cB(\bR^d)$-measurable for each $t\in[0,T]$, $K_T\in L^2({\Omega},{\sF}_T;L^2)$, $\psi+Du\sigma\in\cL^2((L^2)^m)$
   and

  (i)
  $$
  u(t,x+X^{\sigma}_t)=K_T(x)-K_t(x)-\int_t^T \!\! (\psi+Du\sigma)(s,x+X^{\sigma}_s)\,dW_s,\ {\bP}\otimes {dx}\text{-a.e.}
  $$
  for each $t\in[0,T]$. The processes $K$ and $\psi$ are uniquely determined by those properties. Moreover, there hold the following relations:

  (ii)
  \begin{equation*}
  \begin{split}
  &E\left[\|u(t)\|^2+\int_t^T\!\!\|(\psi+Du\sigma)(s)\|^2\,ds\right]
  =E\int_{\bR^d}\!\! (K_T(x)-K_t(x))^2\,dx,\quad \forall\,t\in[0,T];
  \end{split}
  \end{equation*}

  (iii) for any $(\varphi,t)\in \mathcal{D}_T\times [0,T]$,
  \begin{equation*}
  \begin{split}
    &\langle u(t),\,\varphi(t)\rangle
    +\!\!\int_t^T\!\!\! \left( \langle \sigma\psi,\,D \varphi \rangle -\frac{1}{2}\langle u,\,\text{tr}\left(\sigma\sigma'D^2\varphi\right)\rangle +\langle u,\,\partial_s \varphi \rangle\right)(s)\,ds +\!\int_t^T\!\!\! \langle\varphi(s), \psi(s)\,dW_s  \rangle\\
    &=\mu(\varphi1_{[t,T]})=\, \int_t^T\int_{\bR^d}\varphi(s,x)\mu(ds,dx),\quad \text{a.s.},
  \end{split}
  \end{equation*}
  where $\mu$ is the random measure $\mu:\Omega\rightarrow \mathcal{M}([0,T]\times\bR^d)$
  $$
  \textrm{(iv) }~~\quad\quad\quad\quad \quad~~~~~~~~~
  \mu(\varphi 1_{[t,T]})=\int_{\bR^d}\int_t^T\! \varphi(s,x+X^{\sigma}_s)\,dK_s(x)dx,\ \,\varphi\in \mathcal{D}_T, \, \text{a.s.},\quad\quad~~~~~~
  $$
  with $\mathcal{M}([0,T]\times\bR^d)$ denoting the set of all the Radon measures on $[0,T]\times\bR^d$.
\end{thm}

\begin{proof}
   For each $n\in\bN^+$, let $(u_n,\psi_n)\in \cS^2(L^2)\times \mathcal{L}^2((H_2^{-1})^m)$ be the weak solution of BSPDE:
   \begin{equation}\label{BSPDE-prf-thm-potl}
  \left\{\begin{array}{l}
  \begin{split}
    -du_n(t)&=\left[
  \text{tr}\left(\frac{1}{2}\sigma \sigma' D^2 u_n+\sigma D\psi_n\right)(t)+n(u-u_n)(t)\right]\,dt- \psi_n(t)\,dW_t;\\
    u_n(T)  & = 0.
    \end{split}
  \end{array}\right.
 \end{equation}
    In view of \eqref{P-BSPDE-repn} and \eqref{Eq Stoch-Potential}, we have
    \begin{align}
    u_n(t,x)
    =\int_t^Tne^{-n(s-t)}P^{\sigma}_{s-t}u(t,x)\,ds
    \leq\int_t^Tne^{-n(s-t)}u(t,x)\,ds
    \leq& \,u(t,x)\nonumber
    \end{align}
    and
    \begin{align*}
    f_n(t,x):=n(u-u_n)(t,x)=n\int_t^{\infty}ne^{-n(s-t)}(u(t,x)-P^{\sigma}_{s-t}u(t,x))\,ds\geq 0.
    \end{align*}
    Therefore, $0\leq u_n\leq u$. In view of Proposition \ref{prop-BSPDE-Flow}, we have $\bP\otimes dx$-a.e., 
\begin{align}
u_n(t,x+X^{\sigma}_t)=E_{\sF_t}\int_t^T ne^{-n(s-t)}u(s,x+X^{\sigma}_s)\,ds,\quad 0\leq t \leq T,\label{eq-bsde-rep-un}
\end{align}    
  and by the comparison principle for BSDEs, $u_n(t,x+X^{\sigma}_t)$ converges increasingly for every $t\in[0,T]$, $\bP\otimes dx$-a.e. 
%
   
   Noticing that the trajectories of $u_n(t,x+X^{\sigma}_t)$ (by Proposition \ref{prop-quasi-cont-BSPDE})
    and $u(t,x+X^{\sigma}_t)$ are continuous, we have 
    $$
    \lim_{n\rightarrow \infty}\int_t^T ne^{-n(s-t)}u(s,x+X^{\sigma}_s)\,ds=u(t,x+X^{\sigma}_t),\quad \forall\,t\in[0,T],\,\,\bP\otimes dx\text{-a.e.},
    $$
    which together with relation \eqref{eq-bsde-rep-un} implies by Dini's theorem and Lebesgue's domination convergence theorem
    \begin{align}
    \lim_{n\rightarrow\infty} E\int_{\bR^d}\sup_{t\in[0,T]}|u_n(t,x+X^{\sigma}_t)-u(t,x+X^{\sigma}_t)|^2  \,dx=0.  \label{eq-thm-pots-conv-sup-u}
    \end{align}

    Setting $K_t^n(x)=\int_0^t f_n(s,x+X^{\sigma}_s)\,ds$, we have
    {\small
    \begin{align}
    u_n(t,x+X^{\sigma}_t)=(K_T^n-K_t^n)(x)
    -\!\int_t^T\!\!\!(\psi_n+ Du_n\sigma)(s,x+X^{\sigma}_s)\,dW_s.\label{eq-thm-pots-reltn}
    \end{align}
    }
    Thus,
    {\small
    \begin{align}
      &E\int_{\bR^d} (K_T^n-K_t^n)^2(x)\,dx \nonumber\\
      =\,& E\int_{\bR^d} \Big(u_n(t,x+X^{\sigma}_t)
    -\!\int_t^T\!\!\!(\psi_n+Du_n\sigma)(s,x+X^{\sigma}_s)\,dW_s \Big)^2\,dx\nonumber\\
    =\,& E\bigg[\|u_n(t)\|^2+\!\!\int_t^T\!\!\!\|(\psi_n+Du_n\sigma)(s)\|^2  \,ds    \bigg],
    \label{eq-identity-thm-sto-potent}
    \end{align}
    }
    and for the differences, there holds a similar relation. In particular, we have
    {\small
    \begin{align}
    &E\int_{\bR^d}
     |\delta_{nk}K_T(x)|^2dx
    =\, E\bigg[\|\delta_{nk}u(0)\|^2
    +\!\!\int_0^T\!\!\!\|(\delta_{nk}\psi+ D\delta_{nk}u\sigma)(s)\|^2  \,ds    \bigg],\label{eq-thm-post-A}
    \end{align}
    }
 where for $n,k\in\bN^+$,
 $$
 \left(\delta_{nk}u,\delta_{nk}\psi,\delta_{nk}K\right)
 :=
 (u_n,\psi_n,K^n)
 -
 (u_k,\psi_k,K^k).
 $$
On the other hand, It\^o's formula yields
{\small
\begin{align}
&|\delta_{nk}u(t,x+X^{\sigma}_t)|^2
+\int_t^T
|(\delta_{nk}\psi+D\delta_{nk}u\sigma)(s,x+X^{\sigma}_s)|^2
\,ds\nonumber\\
=&
2\int_t^T\delta_{nk}u(s,x+X^{\sigma}_s)\,d \delta_{nk}K_s(x)
-2\int_t^T\delta_{nk}u(s,x+X^{\sigma}_s)(\delta_{nk}\psi+D\delta_{nk}u\sigma)(s,x+X^{\sigma}_s)\,dW_s\label{eq-ito-delta-nk}
\end{align}
}
and
\begin{align}
&E\|u_n(t)\|^2+E\int_t^T \|(\psi_n+Du_n\sigma)(s)\|^2\,ds
\nonumber\\
=\,&E\int_{\bR^d}|u_n(t,x+X^{\sigma}_t)|^2dx
+E\int_{\bR^d}\int_t^T
|(\psi_n+Du_n\sigma)(s,x+X^{\sigma}_s)|^2
\,dsdx\nonumber\\
=\,&
2E\int_{\bR^d}\!\int_t^T\!\!  u_n(s,x+X^{\sigma}_s)\,d K^n_s(x)dx
\nonumber\\
\leq\,& 2E\int_{\bR^d} \sup_{s\in [t,T]}|u_n(s,x+X^{\sigma}_s)|^2\,dx+\frac{1}{2} E\int_{\bR^d} (K_T^n-K_t^n)^2(x)\,dx.\label{eq-ito-un}
\end{align}

Putting \eqref{eq-thm-pots-conv-sup-u}, \eqref{eq-identity-thm-sto-potent} and \eqref{eq-ito-un} together, we obtain
\begin{align}
\sup_{n\in\bN^+}\left\{
E\int_{\bR^d} |K_T^n|^2(x)\,dx + E\int_0^T\|( Du_n\sigma+\psi_n)(s)\|^2ds
\right\}
<\infty.\label{eq-n-bdness}
\end{align}
Without any loss of generality, let $n>k$. Noticing that $(u_n-u_k)(f_n-f_k)\leq (u_n-u_k)f_n$, one has
\begin{align*}
&E\int_{\bR^d}\left|\int_t^T\delta_{nk}u(s,x+X^{\sigma}_s)\,d \delta_{nk}K_s(x)\right|\,dx\\
\leq\,& \left(E\int_{\bR^d} \sup_{s\in [t,T]}|\delta_{nk}u(s,x+X^{\sigma}_s)|^2\,dx\right)^{1/2}
 \left(E\int_{\bR^d} |K^n_T(x)|^2\,dx\right)^{1/2}
\end{align*}
which by the boundedness estimate \eqref{eq-n-bdness} converges to zero as $n$ tends to infinity.
Then it follows from \eqref{eq-ito-delta-nk} that
    \begin{align}
    &E\left[ \int_0^T\!\!\!\Big(\|D\delta_{nk}u\sigma(s)+\delta_{nk}\psi(s)\|^2  \Big)\,ds \right]
    \rightarrow 0,\quad \text{as }n,k\rightarrow\infty
    \end{align}
    which together with relations \eqref{eq-thm-pots-conv-sup-u} and \eqref{eq-thm-post-A} implies
    $$
    E\int_{\bR^d}
     |\delta_{nk}K_T(x)|^2dx\rightarrow 0,\quad\text{as }n,k\rightarrow \infty.
    $$
In view of relation \eqref{eq-thm-pots-reltn}, by Doob's inequality one further has
    \begin{align}
    &E\int_{\bR^d}\sup_{s\in[0,T]}|\delta_{nk}K_s(x)|^2\,dx
    \nonumber\\
    \leq\, &C\left\{
    E\int_{\bR^d} \sup_{s\in [0,T]}|\delta_{nk}u(s,x+X^{\sigma}_s)|^2\,dx
+E \int_0^T\!\!\|D \delta_{nk}u\sigma(s)+\delta_{nk}\psi(s)\|^2 \,ds
\right\}    \nonumber\\
    &\longrightarrow 0,\quad \text{as }n,k\rightarrow \infty, \label{eq-thm-pots-conv-A}
    \end{align}

    Denote by $K$ and $\psi$ the limits of $\{K^n\}$ and $\{\psi_n\}$ respectively. In view of relations \eqref{eq-thm-pots-reltn} and \eqref{eq-identity-thm-sto-potent}, passing to the limit we deduce (i) and (ii).

    As for (iii), the relation holds for the triple $(u_n,v_n,K^n)$ for each $n$, i.e., for any $\varphi\in\mathcal{D}_T$
{\small
\begin{align}
    &\langle u_n(t),\,\varphi(t)\rangle
    +\!\!\int_t^T\!\!\! \Big( \langle\sigma\psi_n,\,D \varphi \rangle-\frac{1}{2}\langle u_n,\,\text{tr}\left(\sigma\sigma'D^2\varphi\right) \rangle  +\langle u_n,\,\partial_s \varphi \rangle\Big)(s)\,ds + \int_t^T\!\!\! \langle\varphi(s), \psi_n(s)\,dW_s  \rangle
    \nonumber\\
    &=\, \int_t^T\int_{\bR^d}\varphi(s,x)f_n(s,x)\,dxds=\int_{\bR^d}\int_t^T\varphi(s,x+X^{\sigma}_s)\,dK^n_s(x)dx,\quad \text{a.s.}, \,\,\forall\, t\in[0,T].
  \label{eq-thm-pots-iii-n}
  \end{align}
  }
  Applying It\^o's formula, one has
  \begin{align*}
  &d\varphi(s,x+X^{\sigma}_s)\\
  =&
  \left(
  	\partial_s\varphi+\frac{1}{2}\text{tr}\left(\sigma\sigma'D^2\varphi\right)
  \right)(s,x+X_s^{\sigma})\,ds+D\varphi(s,x+X^{\sigma}_s)\sigma_s\,dW_s
  \end{align*}
  and
  \begin{align*}
   &\int_t^T \varphi(s,x+X^{\sigma}_s) \,d(K_s-K^n_s)(x)\\
   &=\varphi(T,x+X^{\sigma}_T)(K_T-K^n_T)(x)-\varphi(t,x+X^{\sigma}_t)(K_t-K^n_t)(x)
   -\int_0^T\!\!(K_s-K^n_s)(x)\,d\varphi(s,x+X^{\sigma}_s).
  \end{align*}
  Then in view of \eqref{eq-thm-pots-conv-A}, it is easy to get
%
   \begin{align*}
   &E\left| \int_{\bR^d} \int_t^T \varphi(s,x+X^{\sigma}_s) \,d(K_s-K^n_s)(x)\,dx  \right|
    \rightarrow 0, \quad \text{as }n\rightarrow\infty.
   \end{align*}
   Passing to the limit with $n\rightarrow \infty$ in \eqref{eq-thm-pots-iii-n}, we prove (iii).

  From Doob-Meyer decomposition theorem we conclude the uniqueness of the pair $(K,\,\psi)$.  
  
  Finally, let us consider the converse. First, we verify directly the nonnegativity of $u$, relation \eqref{Eq0 Stocha-Potential} and $\lim_{t\rightarrow T}u(t,\cdot)=0$ in $L^2(\bR^d)$ a.s.  Let $(u_n,\psi_n)$ be the solution of BSPDE \eqref{BSPDE-prf-thm-potl} and put $$(Y_t(x),Y^n_t(x),Z^n_t(x))=(u,u_n,\psi_n+Du_n\sigma)(t,x+X_t^{\sigma}).$$ Then $(Y^n_t(x),Z^n_t(x))$ satisfies BSDE
   \begin{align*}
   Y^n_t(x)=\int_t^Tn(Y_s(x)-Y^n_s(x))\,ds-\int_t^T\,Z^n_s(x)\,dW_s,
   \end{align*}
   and since $Y_s(x)$ is a supermartingale,
   $$
   0\leq Y^n_t(x)=E_{\sF_t}\int_t^Tne^{-n(s-t)}Y_s(x)\,ds\leq Y_t(x).
   $$
   By the penalization procedure for the reflected BSDE  \cite[Page 719-723]{El_Karoui-reflec-1997}, $Y^n$ converges up to $Y$. Taking into account the $\sigma$-quasi-continuity of $u_n$ and the continuity of $Y$, we have further
   $$\lim_{n\rightarrow\infty}E\int_{\bR^d}\sup_{t\in[0,T]}|u(t,x+X^{\sigma}_{t})-u_n(t,x+X^{\sigma}_{t})|^2\,dx=0.$$
   Then
   by Lemma \ref{lem-S-quasi-contin}, $u$ is $\sigma$-quasi-continuous,
   and by relation (i), one has
   \begin{align*}
   P^{\sigma}_ru(t,x+X^{\sigma}_t)=E_{{\sF}_t}u(t+r,x+X^{\sigma}_{t+r})
   =&E_{{\sF}_t}K_T(x) -E_{{\sF}_t}K_{t+r}(x)\\
   \leq&
   E_{{\sF}_t}K_T(x) -K_t(x).
   \end{align*}
   Thus, $P^{\sigma}_ru(t,x+X^{\sigma}_t)\leq u(t,x+X^{\sigma}_t)$ a.e., and there holds relation \eqref{Eq Stoch-Potential}.
   Hence, $u$ is a regular $\sigma$-potential.
\end{proof}

\begin{rmk}\label{rmk-potentl-est-K}
   Thanks to Hahn-Banach theorem and the denseness of $\cD_T$ in the space of continuous functions on $Q$, there is a unique random Radon measure satisfying relation (iii) of Theorem \ref{thm stoch-Potential}. In the following,  we also say that $u$ is a regular $\sigma$-potential associated with couple $(\psi,\mu)$.  Combining relations \eqref{eq-identity-thm-sto-potent} and \eqref{eq-ito-un} and passing to the limits, one gets
  $$
  E|K_T(x)|^2\leq C E\sup_{t\in[0,T]}|u(t,x+X_t^{\sigma})|^2\leq C |u(0,x)|^2,\quad dx\text{-a.e.}
  $$
  with the constants $C$s being independent of $\sigma$,  where the second inequality comes from the supermartingale property of $u(t,x+X_t^{\sigma})$.
%

In addition, as $E\int_{\bR^d} |K_T(x)|^2\, dx<\infty$, for any random field $\phi\in \mathcal{L}^2(L^2(\bR^d))$ satisfying
  $$\phi(t,x+X^{\sigma}_t)\textrm{ is continuous }\mathbb{P}\otimes dx \textrm{-a.e., and }
  E\int_{\bR^d}\sup_{t\in[0,T]}|\phi(t,x+X^{\sigma}_t)|^2\,dx<\infty,$$
  $\mu(\phi)$ makes sense by relation (iv).
\end{rmk}

When $u$ is a deterministic function on $Q$ and $\sigma\sigma'\equiv \mathbb{I}^{d\times d}$, then the approximating BSPDE \eqref{BSPDE-prf-thm-potl} becomes the following deterministic parabolic PDE
   \begin{equation*}
  \left\{\begin{array}{l}
  \begin{split}
    -\partial_tu_n &=
  \frac{1}{2}\Delta u_n+n(u-u_n);\\
    u_n(T)  & = 0.
    \end{split}
  \end{array}\right.
  \end{equation*}
As a result, one has $\psi=0$. One sees that Theorem \ref{thm stoch-Potential} generalizes the classical regular potential in the backward stochastic framework. We refer to \cite[Theorem 2]{MatoussiStoica2010} for the BSDE representation for classical regular potentials, and see  \cite{blumenthal-1968-Markov-Potential,fukushima-2010-dirichlet} for general theory on potentials.

\begin{prop}\label{prop-equiv-potentl}
Let $u\in\cS^2(L^2)$ be $\sigma$-quasi-continuous. Then $u$ is a regular $\sigma$-potential if and only if there exist a random Radon measure $\mu$ and random field $\psi\in\cL^2(H^{-1}_2)$ such that 
for any  $(\varphi,t)\in \mathcal{D}_T\times [0,T]$,
  {\small
  \begin{align}
    &\int_{\bR^d}\int_t^T\!\varphi(t,x)\,\mu(dt,dx)-\!\int_t^T\!\!\! \langle\varphi(s), \psi(s)\,dW_s\rangle\nonumber\\
    =&\langle u(t),\,\varphi(t)\rangle
    +\!\!\int_t^T\!\!\! \left( \langle \sigma\psi,\,D \varphi \rangle -\frac{1}{2}\langle u,\,\text{tr}\left(\sigma\sigma'D^2\varphi\right)\rangle +\langle u,\,\partial_s \varphi \rangle\right)(s)\,ds,\quad \text{a.s.}\label{eq-prop-equiv}
    \end{align}
    }
\end{prop}
\begin{proof}
By (iii) of Theorem \ref{thm stoch-Potential}, it is sufficient to prove the converse.
For each $(y,\phi)\in \bR^d\times C_c^{\infty}(\bR^d)$ and $0\leq t\leq \tilde{t}\leq T$, choosing $\phi\in C_c^{\infty}(\bR^d)$ and $h\in C_c^{\infty}(\bR)$ with $h1_{[-1,T+1]}=1_{[-1,T+1]}$, and  applying relation \eqref{eq-prop-equiv} to test function $h(s)\phi(x)$, we have
\begin{align*}
&\langle u(\tilde t),\,\phi(\cdot-y)\rangle+\int_{\bR^d}\int_t^{\tilde t}\phi(x-y)\,\mu(ds,dx)-\!\int_t^{\tilde t}\!\!\! \langle\phi(\cdot-y), \psi(s)\,dW_s\rangle\\
&=\langle u(t),\,\phi(\cdot-y)\rangle
    +\!\!\int_t^{\tilde t}\!\!\! \left( \langle \sigma\psi(s),\,D \phi (\cdot-y)\rangle -\frac{1}{2}\langle u(s),\,\text{tr}\left(\sigma\sigma'D^2\phi(\cdot-y)\right)\rangle \right)\,ds ,\,\,\text{a.s.}
\end{align*}
In particular, we have $\langle u(T,\cdot),\phi\rangle =0$ a.s., which together with the arbitrariness of $\phi$ implies $u(T,\cdot)=0$, $\bP\otimes dx$-a.e. In a similar way to the proof of \cite[Theorem 1.1]{kunita1981some} for the It\^o-Kunita formula, we have
\begin{align}
&\langle u(\tilde t),\,\phi(\cdot-X^{\sigma}_{\tilde t})\rangle+\int_{\bR^d}\int_t^{\tilde t}\phi(x-X^{\sigma}_s)\,\mu(ds,dx)-\!\int_t^{\tilde t}\!\!\! \langle\phi(\cdot-X^{\sigma}_s), \psi(s)\,dW_s\rangle\nonumber\\
&=\langle u(t),\,\phi(\cdot-X^{\sigma}_t)\rangle
    -\!\int_{t}^{\tilde t}\langle u(s),\, D\phi(\cdot-X^{\sigma}_s)\sigma_s\,dW_s\rangle,  \,\text{a.s.}
    \label{eq-ito-kunita}
\end{align}
To the end, we take $\phi\geq 0$. Then $\left\{ \langle u( t,\cdot+X^{\sigma}_t),\,\phi\rangle\right\}_{t\in[0,T]}$ is a continuous nonnegative supermartingale. Hence, for any $0\leq t_0\leq t_1\leq T$, it holds
$$
\langle u( t_0,\cdot+X^{\sigma}_{t_0}),\,\phi\rangle
\geq E_{\sF_{t_0}} \langle u( t_1,\cdot+X^{\sigma}_{t_1}),\,\phi\rangle ,\quad \text{a.s.}
$$
which together with the arbitrariness of $\phi$ implies
$$u(t_0,x+X^{\sigma}_{t_0}) \geq E_{\sF_{t_0}}  u( t_1,\cdot+X^{\sigma}_{t_1}) , \quad \bP\otimes dx\text{-a.e.}
$$
Obviously, relation \eqref{Eq Stoch-Potential} holds and in view of the $\sigma$-quasi-continuity of $u$, for almost every $x\in\bR^d$, $\left\{u(t,x+X_t^{\sigma})\right\}_{t\in[0,T]}$ is a continuous nonnegative supermartingale. Then it follows that
\[
E\int_{\bR^d} \sup_{t\in[0,T]} |u(t,x+X^{\sigma}_t)|^2 \,dx
\leq C \|u(0)\|^2 <\infty.
\]
Hence, $u$ is a regular $\sigma$-potential.
\end{proof}
In the above proof, the verification for relation \eqref{eq-ito-kunita} is so similar to that of \cite[Theorem 1.1]{kunita1981some}  that we omit it. In fact, compared with \cite[Pages 119-121, proof of Theorem 1.1]{kunita1981some} it is sufficient to notice
\begin{align*}
&\left|\sum_{k=0}^{n-1}\sum_{j=1}^d\int_{t_k}^{t_{k+1}}(X^{\sigma}_{t_{k+1}}-X^{\sigma}_{t_{k}})^j\int_{\bR^d}\partial_{x^j}\phi(x-X^{\sigma}_{t_k})\,\mu(dt,dx)
\right|
\\
&\leq\,C \|D\phi\|_{(L^{\infty})^d}
\sum_{k=0}^{n-1}\left|X^{\sigma}_{t_{k+1}}-X^{\sigma}_{t_{k}}\right|
\int_{t_k}^{t_{k+1}}\int_{\bR^d}1_{\{|x|\leq \max_{s\in[0,t]}|X^{\sigma}_s|+1\}}(x)\,\mu(dt,dx)
\\
&\rightarrow 0 \text{ a.s., as } |\Delta_n|\rightarrow 0,
\end{align*}
where for each $n\in\bN^+$, $\Delta_n=\{0=t_0<t_1<\dots<t_n=T\}$ is a partition of $[0,T]$, $|\Delta_n|=\max_{0\leq k\leq n-1}|t_{k+1}-t_k|$ and without any loss of generality, $\phi$ is  supported in the unit ball of $\bR^d$ centered at the origin.

Furthermore, from relation \eqref{eq-ito-kunita} one can derive immediately through the standard denseness arguments the following representation for the regular $\sigma$-potential via associated random Radon measure.

\begin{cor}\label{cor-regul-Meas-to-potent}
  Let $u$ be a regular $\sigma$-potential and $\mu:\Omega \rightarrow \mathcal{M}([0,T]\times \bR^d)$ a random Radon measure such that relation (iii) holds. Then one has
  \begin{equation}\label{eq lem-stoch-Potential}
    \begin{split}
    &\langle\phi,\,u(t)\rangle
    =E_{\sF_t}\int_t^T\int_{\bR^d}\phi(y-X^{\sigma}_{s}+X^{\sigma}_{t})\mu(dy,ds),
    \end{split}
  \end{equation}
  for each $\phi\in L^2(\bR^d)$ and $t\in[0,T]$.
\end{cor}


We are ready to introduce the family of random measures involving in the notion of the weak solution for stochastic HJB equations.

\begin{defn}\label{def-regul-Stoch-meas}
A nonnegative random Radon measure $\mu:\Omega \rightarrow \mathcal{M}([0,T]\times \bR^d)$ is called regular $\sigma$-measure provided that there exists a regular $\sigma$-potential $u$ such that relation (iii) of Theorem \ref{thm stoch-Potential} is satisfied.
\end{defn}

In view of Definition \ref{def-regul-Stoch-meas}, we see that each regular $\sigma$-measure corresponds to a regular $\sigma$-potential $u$ such that relation (iii) of Theorem \ref{thm stoch-Potential} holds. On the other hand, from Corollary \ref{cor-regul-Meas-to-potent}, we conclude that the corresponding regular $\sigma$-potential can be precisely expressed via \eqref{eq lem-stoch-Potential} in terms of the measure. Therefore, the correspondence between the regular $\sigma$-potential and regular $\sigma$-measure is a bijection. Moreover, by Theorem \ref{thm stoch-Potential} and Proposition \ref{prop-equiv-potentl}, the regular $\sigma$-potential as well as the regular $\sigma$-measure is equivalently characterized via BSDE and in a variational way respectively.

%



\subsection{Monotonic convergence theorem}

\begin{prop}\label{prop-monotone-potential}
  Let $\{u_n;n\in\mathbb{N}^+\}$ be a sequence of regular $\sigma$-potentials converging up to some $u$. Assume further that $u(t,x+X^{\sigma}_t)$ is $\mathbb{P}\otimes dx$-a.e. continuous  with
  $$\int_{\bR^d}E  \sup_{t\in [0,T]}|u(t,x+X^{\sigma}_t)|^2  \,dx<\infty.$$
   Then $u$ is a regular  $\sigma$-potential.
\end{prop}
\begin{proof}
First, Dini's Theorem yields that
\begin{align}
\lim_{n\rightarrow \infty}\int_{\bR^d}E  \sup_{t\in [0,T]}|u_n(t,x+X^{\sigma}_t)-u(t,x+X^{\sigma}_t)|^2 \,dx=0,
\end{align}
from which it follows by Lemma \ref{lem-S-quasi-contin} that $u$ is $\sigma$-quasi-continuous. Let $(\psi_n,K^n)$ be the couple associated with $u_n$. Then
\[
u_n(t,x+X^{\sigma}_t)=K_T^n(x)-K^n_t(x)-\int_t^T(\psi_n+ Du_n\sigma)(s,x+X^{\sigma}_s)\,dW_s,\quad t\in[0,T],
\]
and
\begin{align}
E\int_{\bR^d}|K_T^n(x)|^2\,dx=\|u_n(0)\|^2+E\int_0^T \|\psi_n(s)+ D u_n\sigma(s)\|^2  \,ds.
\label{eq-prop-monot-1}
\end{align}
By It\^o's formula, we have
\begin{align}
&|u_n(t,x+X^{\sigma}_t)|^2+E_{\sF_t}\int_t^T \left|(\psi_n+Du_n\sigma)(s,x+X^{\sigma}_s)\right|^2ds
\nonumber\\
=&\,2E_{\sF_t} \int_t^Tu_n(s,x+X^{\sigma}_s)\,dK_s^n(x) \nonumber\\
\leq& \frac{1}{\eps}
E_{\sF_t}\int_{\bR^d}
\sup_{s\in[t,T]}|u_n(s,x+X^{\sigma}_s)|^2dx
+\eps E_{\sF_t}\int_{\bR^d} |K^n_T(x)-K^n_t(x)|^2dx,\ \forall\, \eps>0\nonumber
\end{align}
which together with \eqref{eq-prop-monot-1} implies
\begin{align}
E
\int_0^T\|D u_n\sigma(s)+\psi_n(s)\|^2ds
+E\int_{\bR^d}|K_T^n(x)|^2dx
\leq
C E\int_{\bR^d}\sup_{s\in[0,T]} |u(s,x+X^{\sigma}_s)|^2dx\label{eq-prop-monotone-bound}
\end{align}
with $C$ being independent of $n$. For $n> l$, putting $(\delta_{nl}u,\delta_{nl}\psi,\delta_{nl}K)=(u_n-u_l,\psi_n-\psi_l,K^n-K^l)$ and applying It\^o's formula again, we have by the monotonicity of $u_n$ and \eqref{eq-prop-monotone-bound}
\begin{align}
&\|\delta_{nl}u(0)\|^2+ E\int_0^T \| D\delta_{nl}u\sigma(s)+\delta_{nl}\psi(s)\|^2  ds\nonumber\\
=\,&2E\int_Q\delta_{nl}u(s,x+X^{\sigma}_s)\,d\left(K^n_s-K^l_s\right)(x)dx\nonumber\\
\leq\,&
2E\int_Q\delta_{nl}u(s,x+X^{\sigma}_s)\,dK^n_s(x)dx\nonumber\\
\leq\,&2
\left(E\int_{\bR^d}\sup_{s\in[0,T]}|\delta_{nl}u(s,x+X^{\sigma}_s)|^2\right)^{1/2}\left( E\int_{\bR^d}|K_T^n(x)|^2dx \right)^{1/2}
\longrightarrow 0, \quad \text{as }n,l\rightarrow \infty.\nonumber
\end{align}
Denote by $\psi+Du\sigma$ the limit of $\psi_n+Du_n\sigma$ in $\cL^2((L^2)^m)$.
 On the other hand, since
\[
u_n(t,x+X^{\sigma}_s)+K_t^n(x)
=
u_n(0,x)+\int_0^t(\psi_n+Du_n\sigma)(s,x+X^{\sigma}_s)\,dW_s,
\]
it follows by Doob's inequality that
\[
E\int_{\bR^d}\sup_{t\in[0,T]}|\delta_{nl}K_t(x)|^2dx
\leq\,C\bigg\{
E\int_{\bR^d}\sup_{t\in[0,T]}|\delta_{nl}u(s,x+X^{\sigma}_s)|^2dx + E\int_0^T\!\!\!\|D \delta_{nl}u\sigma(s)+\delta_{nl}\psi(s)\|^2ds
\bigg\}
\]
which converges to zero as $n$ and $l$ tend to infinity. Denote by $K$ the limit of $K^n$. It follows that
\[
u(t,x+X^{\sigma}_t)=K_T(x)-K_t(x)-\int_t^T (\psi+Du\sigma)(s,x+X^{\sigma}_s)\,dW_s,\quad t\in[0,T].
\]
Hence, $u$ is a regular $\sigma$-potential by Theorem \ref{thm stoch-Potential}.
\end{proof}
In Proposition \ref{prop-monotone-potential}, we write directly the limit of the form $u(t,x+X^{\sigma}_t)$. In fact, this can be derived directly from the norm-equivalence relation \eqref{eq-equiv-norm} and the domination convergence theorem. In what follows, we will omit such kind of arguments for simplicity.

\medskip
   $({\mathcal O}1)$ \it $\xi\in\cL^2(L^2)$ with  $t\mapsto \xi(\omega,t,x+X^{\sigma}_t)$ being $\mathbb{P}\otimes{dx}$-a.e. continuous on $[0,T]$ and
   $$
   E\int_{\bR^d} \sup_{t\in[0,T]} |\xi(t,x+X^{\sigma}_t)|^2  dx<\infty.
   $$

   For each process $\xi$ satisfying $(\mathcal{O} 1)$, we define the Snell envelope $\cE( \xi	)$  by
   \begin{align}
   \cE_t(\xi)=\esssup_{\tau\in\mathscr{J}_t}E_{{\sF}_{t}}\left[ \xi({\tau},x+X^{\sigma}_{\tau}) 1_{\{\tau<T\}}\right] \label{Snell-envelope}
   \end{align}
   where
$$
\mathscr{J}_t=\{ \tau\in\mathscr{J}:\,t\leq \tau\leq T  \},
$$
with $\mathscr{J}$ being the set of all the stopping times dominated by $T$.

\begin{prop}\label{prop-snell-potential}
For each process $\xi$ satisfying $(\mathcal{O} 1)$ with $\xi(T)\leq 0$ $\mathbb{P}\otimes dx$-a.e., one has  $\cE_t( \xi	)=u(t,x+X^{\sigma}_t)$ with $u$ being a regular $\sigma$-potential associated with some couple $(\psi,\mu)$ and
\begin{align}
\int_{Q}(u-\xi)(t,x)\,\mu(dt,dx)=0,\quad \text{a.s.} 
\label{eq-prop-snell-potential}
\end{align}
(For convenience, we will write  $u=\cE(\xi)$ in what follows.)
\end{prop}
\begin{proof}
For each $n\in\bN^+$, let $(u_n,\psi_n)\in\cS^2(L^2)\times\cL^2((H_2^{-1})^m)$ be the unique solution of BSPDE:
   \begin{equation}\label{bspde-linear-lem-snell}
  \left\{\begin{array}{l}
  \begin{split}
    -du_n(t)&=\left[\frac{1}{2}\text{tr}\left(\sigma\sigma'D^2 u_n(t)+2\sigma D\psi_n(t)\right)
    +n\left(u_n(t)-\xi(t)\right)^-\right]\,dt- \psi_n(t)\,dW_t;\\
    u_n(T)  & = 0.
    \end{split}
  \end{array}\right.
 \end{equation}
By Propositions \ref{prop-BSPDE-Flow} and \ref{prop-quasi-cont-BSPDE}, there holds the following BSDE representation
$$
u_n(t,x+X^{\sigma}_t)=\int_t^Tn\left(u_n-\xi\right)^-(s,x+X^{\sigma}_s)\,ds-\int_t^T(\psi_n+D u_n\sigma)(s,x+X^{\sigma}_s)\,dW_s, \,\,t\in[0,T],
$$
and $u_n$ is $\sigma$-quasi-continuous. Therefore, $u_n$ is a regular $\sigma$-potential. On the other hand,  by the penalization method for the reflected BSDE  \cite[Page 719-723]{El_Karoui-reflec-1997}, $u_n$ converges up to $u$ with $u(t,x+X^{\sigma}_t)=\cE_t( \xi	)$ and $u(t,x+X^{\sigma}_t)$ is a continuous process. In view of $(\mathcal{O}1)$ and \eqref{eq-prop-snell-potential}, it is easy to check $E\int_{\bR^d}  \sup_{t\in [0,T]}|u(t,x+X^{\sigma}_t)|^2 dx<\infty$. Hence, $u$ is a regular $\sigma$-potential associated with some couple $(\psi,\mu)$ by Proposition \ref{prop-monotone-potential}, and in view of  (iii) and (iv) in Theorem \ref{thm stoch-Potential}, we further obtain \eqref{eq-prop-snell-potential} from the solution for reflected BSDEs.
\end{proof}

An immediate consequence of Proposition \ref{prop-snell-potential} is the following corollary, which shows that the regular $\sigma$-potential can be charaterized as its own Snell envelope. 

\begin{cor}\label{cor-optim-stop}
Under the same hypothesis of Proposition \ref{prop-snell-potential}, assume further that $\xi(t,x+X^{\sigma}_t)$ is a supermartingale for almost every $x\in\bR^d$ with $\xi(T)=0$. Then one has $\cE(\xi)=\xi$.
\end{cor}

\begin{rmk}\label{rmk-DRBSPDE}
If instead of \eqref{Snell-envelope}, we consider the following optimal stopping problem:
   \begin{align*}
   u(t,x+X^{\sigma}_t)=\esssup_{\tau\in\mathscr{J}_t}E_{{\sF}_{t}}\left[ \int_t^{\tau} H(s,x+X_s^{\sigma})\,ds
   +\xi({\tau},x+X^{\sigma}_{\tau}) 1_{\{\tau<T\}} +\Psi(x+X^{\sigma}_{T})1_{\{\tau=T\}}\right]
   \end{align*}
with $(H,\Psi)\in\cL^2(L^2) \times L^2(\Omega,\sF_T;L^2)$ and $\xi(T)\leq \Psi$ $\mathbb{P}\otimes dx$-a.e. Set $(\tilde{u},\tilde\psi)=\mathbb{S}(\sigma,H,\Psi)$ and $\hat{u}=u-\tilde u$. Then
   \begin{align*}
   \hat{u}(t,x+X^{\sigma}_t)=\cE_t(\xi-\tilde{u})=\esssup_{\tau\in\mathscr{J}_t}E_{{\sF}_{t}}\left[
   (\xi-\tilde{u})({\tau},x+X^{\sigma}_{\tau}) 1_{\{\tau<T\}} \right],
   \end{align*}
   by Proposition \ref{prop-snell-potential}, $\hat{u}$ is a regular $\sigma$-potential associated with some couple $(\hat{\psi},\mu)$, and
    \[\int_{Q}(\hat{u}+\tilde u-\xi)(t,x)\,\mu(dt,dx)=0,\quad \text{a.s.}\]
Putting $\psi=\hat{\psi}+\tilde\psi$, we conclude that the triple $(u,\psi,\mu)$ is a solution of the following degenerate reflected BSPDE:
{\small
\begin{equation}\label{DRBSPDE}
  \left\{\begin{array}{l}
  \begin{split}
  -du(t,x)=\,&\displaystyle \Bigl[\frac{1}{2}\text{tr}\left(\sigma\sigma'D^2 u+2\sigma D\psi\right)(t)+H(t,x)
                \Bigr]\, dt+\mu(dt,x)
           -\psi(t,x)\, dW_{t},\,(t,x)\in Q;\\
    u(T,x)=\, &\Psi(x), \quad x\in\bR^d;\\
    u(t,x)\geq\,& \xi(t,x),\,\,d\mathbb{P}\otimes dt\otimes dx-a.e.;\\
    \int_Q \big( u(t,x)&-\xi(t,x) \big)\,\mu(dt,dx)=0,\,\text{a.s.} \quad \quad \textrm{(Skorohod condition)}
    \end{split}
  \end{array}\right.
\end{equation}
}
in the following sense:

(1) $(u,\psi)\in\cS^2(L^2)\times\cL^2((H^{-1}_2)^m)$ and $\mu$ is a regular $\sigma$-measure;

(2) reflected BSPDE \eqref{DRBSPDE} holds in the weak sense, i.e.,
for each $\varphi\in\mathcal{D}_T$ and $t\in[0,T]$
\begin{equation*}
  \begin{split}
    &\langle u(t),\,\varphi(t) \rangle +\int_t^T\left[\langle u(s),\,\partial_s \varphi(s)   \rangle + \langle D \varphi(s),\,\frac{1}{2}\sigma\sigma' D u(s)
        +\sigma \psi(s) \rangle \right]\,ds\\
    =\,&
    \langle \Psi,\varphi(T) \rangle+\int_t^T\langle H(s),\,\varphi(s) \rangle\,ds
    +\int_{[t,T]\times\bR^d}\varphi(s,x)\mu(ds,dx)-\int_t^T\langle \varphi(s),\,\psi(s)\,dW_s\rangle,\ \text{a.s.} ;
  \end{split}
\end{equation*}

(3) $u$ is $\sigma$-quasi-continuous, $u(t,x)\geq \xi(t,x)$ $\bP\otimes dt\otimes dx$-a.e. and the Skorohod condition of  \eqref{DRBSPDE} holds.

In view of the uniqueness of solution for reflected BSDEs (see \cite{El_Karoui-reflec-1997}), one verifies the uniqueness of solution for reflected BSPDE \eqref{DRBSPDE}. In a similar way to Qiu and Wei \cite{QiuWei-RBSPDE-2013}, one can obtain the existence and uniqueness of solution for reflected BSPDE  \eqref{DRBSPDE} associated with the external force $H$ depending on $u$ and $\psi+Du\sigma$ in a nonlinear fashion. This seems to be new for reflected BSPDEs by dropping the super-parabolicity requirements in \cite{QiuWei-RBSPDE-2013}. For the literature on reflected BSPDEs, we refer to Qiu and Wei \cite{QiuWei-RBSPDE-2013} and references therein.
\end{rmk}


\section{Solvability of stochastic HJB equation \eqref{SHJB}}

\subsection{Associated control problem}

Letting $\cU$ be the admissible control set, we consider the following control problem
\begin{align}
\inf_{\sigma\in\cU}E\left[\int_0^T\!\! f(s,x+X^{\sigma}_s,\sigma_s)\,ds +G(x+X^{\sigma}_T) \right] \label{Control-probm}
\end{align}
subject to $X_t^{\sigma}=\int_0^t\sigma_s\,dW_s$, $t\in[0,T]$. Then the dynamic cost functional is defined by
\begin{align}
J(t,x;\sigma)=E_{\sF_t}\left[\int_t^T\!\! f(s,x+X^{\sigma}_s-X^{\sigma}_t,\sigma_s)\,ds +G(x+X^{\sigma}_T-X_t^{\sigma}) \right],\ \ t\in[0,T] \label{eq-cost-funct}
\end{align}
and the value function is given by
\begin{align}
V(t,x)=\essinf_{\sigma\in\cU}J(t,x;\sigma),\quad t\in[0,T].
\label{eq-value-func}
\end{align}
\begin{rmk}\label{rmk-infn}
For any given $(t,x,\bar\sigma)\in[0,T]\times\bR^d\times \cU$,  set
$$
\mathbb{J}(t,x,\bar\sigma)=\left\{
J(t,x+X^{\bar\sigma}_t;\sigma): J(t,x+X^{\bar\sigma}_t;\sigma)\leq J(t,x+X^{\bar\sigma}_t;\bar\sigma),\,\,\sigma\in\cU
\right\}.
$$
Then $\mathbb{J}(t,x,\bar\sigma)$ is nonempty and for any $J(t,x+X^{\bar\sigma}_t;\tilde\sigma),J(t,x+X^{\bar\sigma}_t;\check\sigma)\in \mathbb{J}(t,x,\bar\sigma)$, putting 
$$
\gamma_s=\bar\sigma_s1_{\{s\in[0,t)\}}+
\left(\tilde\sigma_s 1_{\{J(t,x+X^{\bar\sigma}_t;\tilde\sigma)\leq J(t,x+X^{\bar\sigma}_t;\check\sigma)\}}
+  \check\sigma_s 1_{\{J(t,x+X^{\bar\sigma}_t;\tilde\sigma)> J(t,x+X^{\bar\sigma}_t;\check\sigma)\}}
\right)1_{\{s\in[t,T]\}},
$$
 one has $\gamma\in\cU$ and 
$$
J(t,x+X^{\bar\sigma}_t;\tilde\sigma)\wedge J(t,x+X^{\bar\sigma}_t;\check\sigma)=
J(t,x+X^{\bar\sigma}_t;\gamma)\in \mathbb{J}(t,x,\bar\sigma).
$$
Hence, by \cite[Theorem A.3]{Karatz-Shreve-1998MMF}, there exists  $\{\sigma^n\}_{n\in\bN^+}\subset \cU$ such that
$J(t,x+X^{\bar\sigma}_t;\sigma^n)$ converges decreasingly to $V(t,x+X^{\bar\sigma}_t)$ with probability 1.
\end{rmk}

\begin{lem}\label{lem-value-func}
Under assumption $(\cA 1)$, we have

(i) For any  $(t,x,\sigma)\in[0,T]\times\bR^d\times\cU$, there exists $\bar{\sigma}\in\cU$ such that
$$
E\left[ J(t,x+X^{\sigma}_t;\bar{\sigma})-V(t,x+X^{\sigma}_t)\right]<\eps;
$$

(ii) For each $(\bar{\sigma},x)\in\cU\times \bR^d$, $\left\{J(t,x+X_t^{\bar{\sigma}};\bar{\sigma})-V(t,x+X_t^{\bar{\sigma}})\right\}_{t\in[0,T]}$ is a supermartingale, i.e.,  for any $0\leq t\leq \tilde{t}\leq T$,
\begin{align}
V(t,x+X_t^{\bar{\sigma}})
\leq E_{\sF_t}V(\tilde{t},x+X_{\tilde{t}}^{\bar{\sigma}}) + E_{\sF_t}\int_t^{\tilde{t}}f(s,x+X_s^{\bar{\sigma}},\bar\sigma_s)\,ds,\,\,\,\text{a.s.};\label{eq-vfunc-supM}
\end{align}

(iii) For each $(\bar{\sigma},x)\in\cU\times \bR^d$, $\left\{V(s,x+X_s^{\bar{\sigma}})\right\}_{s\in[0,T]}$ is a continuous process.

(iv) For each $\sigma\in\cU$,
\begin{align}
\int_{\bR^d}E\sup_{t\in[0,T]}|V(t,x+X_t^{\sigma})|^2\,dx<\infty,
\label{eq-iv-lemm-value-funct}
\end{align}
and there exists $L_1>0$ such that for any $\sigma\in\cU$
$$
|V(t,x)-V(t,y)|+|J(t,x;\sigma)-J(t,y;\sigma)|\leq L_1|x-y|^{\alpha},\,\,\,\text{a.s.},\quad \forall\,x,y\in\bR^d.
$$

\end{lem}

\begin{proof}
From Remark \ref{rmk-infn}, assertion (i) follows obviously. Again by Remark \ref{rmk-infn}, there exists $\{\sigma^n\}_{n\in\bN^+}\subset \cU$ such that $J(\tilde t,x+X^{\bar\sigma}_{\tilde t};\sigma^n)$ converges decreasingly to $V(\tilde t,x+X^{\bar\sigma}_{\tilde t})$ with probability 1. Therefore, we have
{\small
\begin{align*}
&E_{\sF_t}
V(\tilde{t},x+X_{\tilde{t}}^{\bar{\sigma}})
+ E_{\sF_t}\int_t^{\tilde{t}}f(s,x+X_s^{\bar{\sigma}},\bar{\sigma}_s)\,ds
\\
=\,&
E_{\sF_t}
\lim_{n\rightarrow \infty} 
J(\tilde t,x+X^{\bar\sigma}_{\tilde t};\sigma^n)
+ E_{\sF_t}\int_t^{\tilde{t}}f(s,x+X_s^{\bar{\sigma}},\bar{\sigma}_s)\,ds
\\
=\,&
\lim_{n\rightarrow \infty} 
E_{\sF_t}J(\tilde t,x+X^{\bar\sigma}_{\tilde t};\sigma^n)
+ E_{\sF_t}\int_t^{\tilde{t}}f(s,x+X_s^{\bar{\sigma}},\bar{\sigma}_s)\,ds
\\
=\,&
\lim_{n\rightarrow \infty}
E_{\sF_t}
\left[
\int_{\tilde{t}}^T
f(s,x+X_{\tilde{t}}^{\bar{\sigma}}+X_s^{\sigma^n}-X_{\tilde{t}}^{{\sigma}^n},\sigma^n_s)\,ds+
\int_t^{\tilde{t}}f(s,x+X_s^{\bar{\sigma}},\bar{\sigma}_s)\,ds+G(x+X_{\tilde{t}}^{\bar{\sigma}}+X_T^{\sigma^n}-X_{\tilde{t}}^{{\sigma^n}})
\right]\\
\geq\,&
\essinf_{\sigma\in\cU}E_{\sF_t}
\left[
\int_{t}^T
f(s,x+X_{t}^{\bar{\sigma}}+X_s^{\sigma}-X_{t}^{{\sigma}},\sigma_s)\,ds +G(x+X_{t}^{\bar{\sigma}}+X_T^{\sigma}-X_{t}^{{\sigma}})
\right]
\\
=\,&V(t,x+X_t^{\bar{\sigma}}),\quad \text{a.s.,}
\end{align*}
}
which verifies \eqref{eq-vfunc-supM} as well as assertion (ii).

 Then we have for $0\leq t\leq \tilde{t}\leq T$,
{\small
\begin{align*}
&E_{\sF_t}\int_t^{\tilde{t}}g(s,x+X_s^{\bar\sigma})\,ds\quad\text{(see assumption $(\cA 1)$)}\\
&\geq E_{\sF_t}\int_t^{\tilde{t}}f(s,x+X_s^{\bar{\sigma}},\bar{\sigma}_s)\,ds
\\
&\geq\,V(t,x+X_t^{\bar{\sigma}})-E_{\sF_t}V(\tilde{t},x+X_{\tilde{t}}^{\bar{\sigma}})\\
&=
\essinf_{\sigma\in\cU}
E_{\sF_t}\left[
\int_{t}^T
f(s,x+X_{t}^{\bar{\sigma}}+X_s^{\sigma}-X_{t}^{{\sigma}},\sigma_s)\,ds +G(x+X_{t}^{\bar{\sigma}}+X_T^{\sigma}-X_{t}^{{\sigma}})
\right]\\
&\quad-E_{\sF_t}\essinf_{\sigma\in\cU}
E_{\sF_{\tilde{t}}}\left[
\int_{\tilde{t}}^T
f(s,x+X_{\tilde{t}}^{\bar{\sigma}}+X_s^{\sigma}-X_{\tilde{t}}^{{\sigma}},\sigma_s)\,ds +G(x+X_{\tilde{t}}^{\bar{\sigma}}+X_T^{\sigma}-X_{\tilde{t}}^{{\sigma}})
\right]
\\
&\geq
\essinf_{\sigma\in\cU}
E_{\sF_t}\left[
\int_{t}^T
f(s,x+X_{t}^{\bar{\sigma}}+X_s^{\sigma}-X_{t}^{{\sigma}},\sigma_s)\,ds +G(x+X_{t}^{\bar{\sigma}}+X_T^{\sigma}-X_{t}^{{\sigma}})
\right]\\
&\quad-\essinf_{\sigma\in\cU}
E_{\sF_t}\left[
\int_{\tilde{t}}^T
f(s,x+X_{\tilde{t}}^{\bar{\sigma}}+X_s^{\sigma}-X_{\tilde{t}}^{{\sigma}},\sigma_s)\,ds +G(x+X_{\tilde{t}}^{\bar{\sigma}}+X_T^{\sigma}-X_{\tilde{t}}^{{\sigma}})
\right]
\\
&\geq
\essinf_{\sigma\in\cU}\bigg\{
E_{\sF_t}\int_{t}^{\tilde{t}}
f(s,x+X_{t}^{\bar{\sigma}}+X_s^{\sigma}-X_{t}^{{\sigma}},\sigma_s)\,ds\\
&\quad\quad\quad
+E_{\sF_t}\int_{\tilde{t}}^T
\left(f(s,x+X_{t}^{\bar{\sigma}}+X_s^{\sigma}-X_{t}^{{\sigma}},\sigma_s)
-f(s,x+X_{\tilde{t}}^{\bar{\sigma}}+X_s^{\sigma}-X_{\tilde{t}}^{{\sigma}},\sigma_s)\right)\,ds
\\
&\quad\quad\quad
+
E_{\sF_t}\left[
G(x+X_{t}^{\bar{\sigma}}+X_T^{\sigma}-X_{t}^{{\sigma}})-G(x+X_{\tilde{t}}^{\bar{\sigma}}+X_T^{\sigma}-X_{\tilde{t}}^{{\sigma}})
\right]
\bigg\}\\
&\geq
-\esssup_{\sigma\in\cU}\bigg\{E_{\sF_t}\int_{\tilde{t}}^TC\Big|\int_t^{ \tilde{t}}(\sigma_r-\bar{\sigma}_r)\,dW_r\Big|^{\alpha}\,ds
+CE_{\sF_t}\Big|\int_t^{ \tilde{t}}(\sigma_r-\bar{\sigma}_r)\,dW_r\Big|^{\alpha}\bigg\}\\
&\geq
-C(\tilde{t}-t)^{\alpha/2}\longrightarrow 0,\quad \text{as }|t-\tilde{t}|\rightarrow 0.
\end{align*}
}
In a similar way, one proves the continuity of $EV(t,x+X^{\bar{\sigma}}_t)$ in $t$, which by the regularity of supermartingale implies the right continuity of $V(t,x+X^{\bar{\sigma}}_t)$. On the other hand, by the BSDE theory, $E_{\sF_t}V(\tilde{t},x+X^{\bar{\sigma}}_{\tilde{t}})$ and $E_{\sF_t}\int_t^{\tilde{t}}g(s,x+X_s^{\bar{\sigma}})\,ds$ are continuous in $t\in[0,\tilde{t}]$, and this together with the above calculations implies the left continuity of $V(t,x+X^{\bar{\sigma}}_t)$ in $t$. Hence, $\left\{V(s,x+X_s^{\bar{\sigma}})\right\}_{s\in[0,T]}$ is a continuous process and we prove assertion (iii).

 As for (iv), relation \eqref{eq-iv-lemm-value-funct}  follows obviously from Proposition \ref{prop-BSPDE-Flow}, Corollary \ref{cor-degenerate-BSPDE} and the fact that $V(t,x+X^{\sigma}_t)\leq J(t,x+X^{\sigma}_t;\sigma)$ for any $\sigma\in\cU$. For any $x,y\in\bR^d$, by definition of the value function, we have
\begin{align*}
|V(t,x)-V(t,y)|&\leq \esssup_{\sigma\in\cU}E_{\sF_t}\bigg[\int_t^T\!\!|f(s,x+X_s^{\sigma}-X_t^{\sigma},\sigma_s)-f(s,y+X_s^{\sigma}-X_t^{\sigma},\sigma_s)|\,ds
\\
&\quad\quad\quad+|G(x+X_T^{\sigma}-X_t^{\sigma})-G(y+X_T^{\sigma}-X_t^{\sigma})|\bigg]
\\
&\leq
L(T-t)|x-y|^{\alpha}+L|x-y|^{\alpha}
\end{align*}
from which one derives the $\alpha$-H\"older continuity of $V(t,x)$ and $J(t,x;\sigma)$ in $x$.
We complete the proof.
\end{proof}


\subsection{Existence and uniqueness of the weak solution for stochastic HJB equation \eqref{SHJB}}


\begin{defn}\label{def-weak-sltn-HJB}
A couple $(u,\psi)\in \cS^2(L^2)\times \cL^2((H^{-1}_2)^m)$ is said to be a weak solution of BSPDE \eqref{SHJB}, if  for each $\sigma\in \cU$, $u$ is $\sigma$-quasi-continuous and there exists a random Radon measure $\mu^{\sigma}$, such that for any $\varphi\in \cD_T$,
\begin{align}
    &\langle u(t),\,\varphi(t)\rangle
    +\!\!\int_t^T\!\! \left( \langle \sigma\psi,\,D \varphi \rangle -\frac{1}{2}\langle u,\,\text{tr}\left(\sigma\sigma'D^2\varphi\right)\rangle +\langle u,\,\partial_s \varphi \rangle\right)(s)\,ds +\!\int_t^T\!\! \langle\varphi(s), \psi(s)\,dW_s  \rangle\nonumber\\
    &=\langle G,\,\varphi(T)\rangle+\int_t^T\langle f(s,\cdot,\sigma_s),\,\varphi(s) \rangle\,ds - \int_t^T\int_{\bR^d}\varphi(s,x)\mu^{\sigma}(ds,dx),\quad t\in[0,T].\label{eq-defn-weaksltn}
\end{align}
and the infimum of $\{\mu^{\sigma}\}_{\sigma\in\cU}$ vanishes in the sense that for each nonnegative $\tilde\varphi\in C_{c}^{\infty}(\bR^d)$ and any $(t,\eps,\sigma)\in [0,T]\times(0,\infty)\times\cU$, there exist a sequence $\{\sigma^i\}\subset\cU$ and a standard partition $\{\zeta^i\}_{i\in\bN^+}$ of unity in $\bR^d$ such that
\begin{align}
E\left[\sum_{i\in\bN^+}\int_{[t,T]\times\bR^d} (\tilde\varphi\zeta^i)(x-X_t^{{\sigma}}+X_s^{\sigma^i}-X_t^{\sigma^i})\,\mu^{\sigma^{[0,t]}\vee \sigma^i}(ds,dx)\right]<\eps,\label{eq-defn-weak-sltn-HJB}
\end{align}
 where
$\sigma^{[0,t]}\vee \sigma^i\in\cU$ with $\left(\sigma^{[0,t]}\vee \sigma^i\right)(s)=\sigma_s 1_{[0,t]}(s)+\sigma^i_s1_{(t,T]}(s)$ for $s\in[0,T]$.

\end{defn}

\begin{rmk}\label{rmk-uniq-diffn}
 In view of Definition \ref{def-weak-sltn-HJB}, we have for any $(\varphi,\sigma)\in\mathcal{D}_T\times \cU$,
\begin{align*}
    &\langle u(0),\,\varphi(0)\rangle
    +\!\!\int_0^t\!\! \left( \langle \sigma\psi,\,D \varphi \rangle -\frac{1}{2}\langle u,\,\text{tr}\left(\sigma\sigma'D^2\varphi\right)\rangle +\langle u,\,\partial_s \varphi \rangle\right)(s)\,ds +\!\int_0^t\!\! \langle\varphi(s), \psi(s)\,dW_s  \rangle\\
    &=\langle u(t),\,\varphi(t)\rangle
    +\int_0^t\langle f(s,\cdot,\sigma_s),\,\varphi(s)\rangle\,ds
    - \int_0^t\int_{\bR^d}\varphi(s,x)\mu^{\sigma}(ds,dx),\quad t\in[0,T].
\end{align*}
Thus $\langle u(t),\,\varphi(t)\rangle$ is a semimartingale. From the Doob-Meyer decomposition theorem and the arbitrariness of $\varphi$, one concludes that $\psi$ is uniquely determined by $u$ in the weak solution for BSPDEs like \eqref{SHJB}. On the other hand, by Proposition \ref{prop-equiv-potentl}, we see that for each $\sigma\in \cU$, $J(t,x;\sigma)-u(t,x)$ is a regular $\sigma$-potential.
%
\end{rmk}

\begin{thm}\label{thm-EU-main}
Under assumption $(\cA 1)$, BSPDE \eqref{SHJB} admits a unique weak solution $(u,\psi)$ with $u$ coinciding with the value function  of \eqref{eq-value-func} for stochastic optimal control problem \eqref{Control-probm}. For this solution,  $\psi+Du\sigma\in\cL^2((L^2)^m)$  for each $\sigma\in\cU$, and  there exists $L_1>0$ such that for any $x,y\in\bR^d$, $\sup_{t\in[0,T]}|u(t,x)-u(t,y)|\leq L_1|x-y|^{\alpha}$ a.s.
\end{thm}
\begin{proof}
\textit{Existence.} Consider the value function $V$ of \eqref{eq-value-func} for stochastic optimal control problem \eqref{Control-probm}. Put $(u^{\sigma},\psi^{\sigma})=\mathbb{S}(\sigma,f,G)$ for $\sigma\in\cU$. One has $u^{\sigma}(t,x)=J(t,x;\sigma)$. In view of assertions (ii), (iii) and (iv) in Lemma \ref{lem-value-func}, from Proposition \ref{prop-snell-potential} and Corollary \ref{cor-optim-stop} we deduce that $u^{\sigma}-V$ is a regular $\sigma$-potential, whose associated random Radon measure is denoted by $\mu^{\sigma}$. Further by Proposition \ref{prop-equiv-potentl}, we conclude the relation \eqref{eq-defn-weaksltn}, where we denote the associated diffusion term of $V$ by $\psi$ which in a similar way to Remark \ref{rmk-uniq-diffn} is uniquely determined by $V$..

  For relation \eqref{eq-defn-weak-sltn-HJB}, we only give the proof for the case $t=0$, since it follows similarly for $t\in(0,T]$.

For each $\eps\in(0,1)$, choose the partition $\{\zeta^i\}_{i\in\bN^+}$ of unity  such that  for each $i\in\bN^+$, $\text{supp\,}\zeta^i\subset B(x^i,\eps)$ for some $x^i\in\bR^d$, where $B(x^i,\eps)$ denotes the open ball of radius $\eps$ centered at $x^i$.
By assertion (i) of Lemma \ref{lem-value-func}, for each $i\in\bN^+$ we take $\sigma^i\in\cU$ such that $|V(0,x^i)-J(0,x^i;\sigma^i)|<\eps$.  By the uniform H\"older continuity of $V(t,x)$ and $J(t,x;\sigma)$ in $x$, one has further $|V(0,x)-J(0,x;\sigma^i)|<(2L_1+1)\eps^{\alpha}$  for any $x\in\bR^d$. On the other hand, in view of Remark \ref{rmk-potentl-est-K} one has
\[
E|K^{\sigma}_T(x)|^2\leq C |V(0,x)-J(0,x;\sigma)|^2,
\]
where $K^{\sigma}$ is the increasing process associated to the regular $\sigma$-measure $\mu^{\sigma}$ (see (iii) of Theorem \ref{thm stoch-Potential}).

For the associated regular $\sigma$-measure $\mu^{\sigma}$, we have for any nonnegative $\varphi\in C_{c}^{\infty}(\bR^d)$,
\begin{align}
E\left[\sum_{i\in\bN^+}\int_Q (\varphi\zeta^i)(x-X_t^{\hat{\sigma}})\mu^{\sigma^i}(dt,dx)\right]
&=
E\left[\sum_{i\in\bN^+}\int_{\bR^d}\int_{0}^T (\varphi\zeta^i)(x)\,dK^{\sigma^i}_t(x)dx\right]
\nonumber\\
&\leq
\sum_{i\in\bN^+}\int_{\bR^d}  \varphi(x) \zeta^i(x) \left(E\left|K^{\sigma^i}_T(x)\right|^2\right)^{1/2}dx\nonumber\\
&\leq
C(2L_1+1)\|\varphi\|_{L^1(\bR^d)}\eps^{\alpha}.
\end{align}
Hence, $(V,\psi)$ is a weak solution to BSPDE \eqref{SHJB}.

\textit{Uniqueness.}
To prove the uniqueness, we need only verify that $u$ coincides with the value function $V$ for any solution $(u,\psi)$ of BSPDE \eqref{SHJB}. In fact, by Definition \ref{def-weak-sltn-HJB} and Proposition \ref{prop-equiv-potentl}, for each $\sigma\in\cU$ one always has $u(t,x)-J(t,x;\sigma)\leq 0$ and $\left\{ J(t,x+X^{\sigma}_t;\sigma)-u(t,x+X_{t}^{\sigma})  \right\}_{t\in[0,T]}$ is a  nonnegative continuous supermartingale. It is obvious that $u(t,x)\leq V(t,x)$, $\bP\otimes dt \otimes dx$-a.e.
On the other hand, in view of \eqref{eq-defn-weak-sltn-HJB}, one has 
\begin{align*}
\langle \varphi,\,\sum_{i\in\bN^+}\zeta^iJ(0,\cdot;\sigma^i)-u(0,\cdot)\rangle
&=
E\left[\sum_{i\in\bN^+}\int_{\bR^d}\int_{0}^T (\varphi\zeta^i)(x)\,dK^{\sigma^i}_t(x)dx\right]
\\
&=E\left[\sum_{i\in\bN^+}\int_Q (\varphi\zeta^i)(x-X_t^{{\sigma}^i})\mu^{\sigma^i}(dt,dx)\right]\\
&< \eps.
\end{align*}
It follows that
\[
\langle \varphi,\,V(0,\cdot)\rangle\geq \langle\varphi,\,u(0,\cdot)\rangle > \langle \varphi,\,V(0,\cdot)\rangle -\eps,
\]
which together with the arbitrariness of $(\eps,
\varphi)$ implies $V(0,\cdot)=u(0,\cdot)$. In a similar way, one has further $V(t,\cdot)=u(t,\cdot)$ for any $t\in[0,T]$. This completes the proof.
\end{proof}

  An immediate consequence of Theorem \ref{thm-EU-main} is the dynamic programming principle for the stochastic optimal control problems of non-Markovian type, which was established by Peng \cite[Theorem 6.6, Page 123]{Peng-DPP-1997} with a different method.
\begin{cor}\label{rmk-DPP}
Under assumption $(\cA1)$, there holds for any $0\leq t\leq t+\delta\leq T$ and $\zeta\in L^2(\Omega,\sF_t)$,
\[
V(t,\zeta)=\essinf_{\sigma\in\cU}E_{\sF_t}\left[\int_t^{t+\delta}f(s,\zeta+X^{\sigma}_s-X^{\sigma}_t,\sigma_s)\,ds
+V(t+\delta,\zeta+X^{\sigma}_{t+\delta}-X^{\sigma}_t)
\right],\,\text{a.s.}
\]
\end{cor}


\subsection{On the regularity}

We shall study the partially non-Markovian case and derive certain regular properties. Rewrite the Wiener process $W=(\tilde{W},\bar{W})$ with $\tilde{W}$ and $\bar{W}$ being two mutually independent and respectively, $m_0$ and $m_1$ dimensional  Wiener processes.


Instead of BSPDE \eqref{SHJB}, we consider
{\small
\begin{equation}\label{SHJB-SP}
  \left\{\begin{array}{l}
  \begin{split}
  -du(t,x)=\,&\displaystyle  
  \essinf_{\sigma=(\tilde{\sigma},\bar{\sigma})\in U} \bigg\{\text{tr}\left(\frac{1}{2}(\bar{\sigma} \bar{\sigma}' +\tilde{\sigma}\tilde{\sigma}')D^2 u+\tilde{\sigma} D\psi\right)(t,x)
        +f(t,x,\sigma)
                \bigg\} \,dt\\ &\displaystyle
           -\psi(t,x)\, d\tilde{W}_{t}, \quad
                     (t,x)\in Q;\\
    u(T,x)=\, &G(x), \quad x\in\bR^d.
    \end{split}
  \end{array}\right.
\end{equation}}
Here and in the following, we adopt the decomposition $\sigma=(\tilde{\sigma},\bar{\sigma})$ with $\tilde{\sigma}$ and $\bar{\sigma}$ valued in $\bR^{n\times m_0}$ and $\bR^{n\times m_1}$ respectively for the control $\sigma$, and associated with $(\tilde{W}, \bar{W})$, we take $\Omega=C([0,T];\bR^{m})$, $\tilde{\Omega}=C([0,T];\bR^{m_0})$, $\bar{\Omega}=C([0,T];\bR^{m_1})$, and $\omega=(\tilde{\omega},\bar{\omega})$ with $\omega\in\Omega$, $\tilde{\omega}\in\tilde{\Omega}$ and $\bar{\omega}\in\bar{\Omega}$.

Denote by $\{\tilde{\sF}_t\}_{t\geq0}$ the natural filtration generated by $\tilde{W}$ and augmented by all the
$\bP$-null sets.

\bigskip\medskip
   $({\mathcal A} 2)$ \it $G\in L^{2}(\Omega,\tilde{\sF}_T;L^2)$ and for each $(t,\tilde{v})\in [0,T]\times U$, the   random function
$
  f(\cdot,t,\cdot,\tilde{v}):~\Omega\times\bR^d\rightarrow \bR
$
is $\tilde{\sF}_t\otimes\cB(\bR^d)$-measurable.\rm

First, we shall show a measurability property of the value function $V(t,x)$.
\begin{lem}\label{lem-meas-value-funct}
Under assumptions $(\cA1)$ and $(\cA2)$, the value function $V(t,x)$ defined by \eqref{eq-value-func} is $\tilde{\sF}_t$-measurable for each $(t,x)\in[0,T]\times \bR^d$.
\end{lem}
\begin{proof}
Set
\[
\bH:=\left\{h;h(0)=0,\frac{dh}{dt}\in L^2(0,T;\bR^{m_1})\right\},
\]
which is the Cameron-Martin space associated with the Wiener process $\bar{W}$. For any $h\in\bH$, we define the translation operator $\tau_h:{\Omega}\rightarrow {\Omega}$, $\tau_h((\tilde{\omega},\bar{\omega}))=(\tilde{\omega},\bar{\omega}+h)$ for ${\omega}=(\tilde{\omega},\bar{\omega})\in{\Omega}$. It is obvious that $\tau_h$ is a bijection and that it defines the probability transformation: $\left(\bP\circ \tau_h^{-1}\right)(d\omega)=\exp\{\int_0^T|\frac{dh}{dt}|^2\,dt-\frac{1}{2}\int_0^T\frac{dh}{dt}\,d\bar{W}_t\}\bP(d\omega)$. Fix some $(t,x)\in[0,T]\times\bR^d$. By Girsanov theorem, it follows that for any $\sigma\in\cU$,  $X^{\sigma}(\tau_h)=X^{\sigma(\tau_h)}$ and $J(t,x;\sigma)(\tau_h)=J(t,x;\sigma(\tau_h))$, $\bP$-a.s. Taking into account the fact that $\{\sigma(\tau_h)|\sigma\in\cU\}=\cU$, one gets further that
\[
\left(
\essinf_{\sigma\in\cU}J(t,x;\sigma)
\right)(\tau_h)
=
\essinf_{\sigma\in\cU}J(t,x;\sigma(\tau_h))
= \essinf_{\sigma(\tau_{-h})\in\cU}J(t,x;\sigma)
= \essinf_{\sigma\in\cU}J(t,x;\sigma).
\]
Hence, $V(t,x)(\tau_h)=V(t,x)$ $\bP$-a.s. for any $h\in\bH$. In particular, for any continuous and bounded function $\Phi$,
\begin{align*}
&E\left[  \Phi(V(t,x))\exp\Big\{\int_0^T|\frac{dh}{ds}|^2\,ds-\frac{1}{2}\int_0^T\frac{dh}{ds}\,d\bar{W}_s\Big\}  \right]\\
&=
E\left[  \Phi(V(t,x))(\tau_{h})\exp\Big\{\int_0^T|\frac{dh}{ds}|^2\,ds-\frac{1}{2}\int_0^T\frac{dh}{ds}\,d\bar{W}_s\Big\}  \right]
\\
&=E\left[  \Phi(V(t,x))\right]
E\left[\exp\Big\{\int_0^T|\frac{dh}{ds}|^2\,ds-\frac{1}{2}\int_0^T\frac{dh}{ds}\,d\bar{W}_s\Big\}  \right],
\end{align*}
which together with the arbitrariness of $(\Phi,h)$ implies that $V(t,x)$ is just $\tilde{\sF}_t$-measurable.
\end{proof}
 Denote by $\tilde{\cU}$ the set of all the $\tilde{\sF}_t$-adapted elements of $\cU$. In a similar way to Lemma \ref{lem-meas-value-funct}, one gets the following
\begin{cor}
Under assumptions $(\cA1)$ and $(\cA2)$,  the cost functional $J(t,x;\sigma)$ defined by \eqref{eq-cost-funct} is $\tilde{\sF}_t$-measurable for each $(t,x,\sigma)\in[0,T]\times \bR^d\times\tilde{\cU}$.
\end{cor}
Under assumptions $(\cA1)$ and $(\cA2)$, let $(u,\psi)$ be the unique solution of BSPDE \eqref{SHJB} in Theorem  \ref{thm-EU-main}. Then by Lemma \ref{lem-meas-value-funct}, for each $(t,x)\in[0,T]\times\bR^d$, $u(t,x)=V(t,x)$ is just $\tilde{\sF}_t$-measurable, and  making similar arguments as in Remark \ref{rmk-uniq-diffn}, especially, by investigating the case $\sigma_s=\sigma_0\in U$ for any $s\in[0,T]$,  we deduce that  $\psi_i=0$ for $i=m_0+1,\dots,m_0+m_1$. Then BSPDE \eqref{SHJB} can be written into the form of BSPDE \eqref{SHJB-SP}. Similarly, for each $\sigma\in\tilde{\cU}$, letting $(u^{\sigma},\psi^{\sigma})=\mathbb{S}(\sigma,f,G)$, then $\psi^{\sigma}_i=0$ for $i=m_0+1,\dots,m_0+m_1$. In view of Definition \ref{def-weak-sltn-HJB} and Theorem \ref{thm stoch-Potential}, one concludes that for any $\sigma=(\tilde{\sigma},\bar{\sigma})\in\tilde{\cU}$, there holds the gradient estimate $Du\bar{\sigma}\in\cL^2((L^2)^{m_1})$. Hence, we have

\begin{prop}\label{prop-parabolc-BSPDE}
Under assumptions $(\cA1)$ and $(\cA2)$, for the unique solution $(u,\psi)$ of BSPDE \eqref{SHJB}, $u(t,x)$ is just $\tilde{\sF}_t$-measurable for each $(t,x)\in[0,T]\times \bR^d$ and $\psi_i=0$ for $i=m_0+1,\dots,m_0+m_1$, and for any $\sigma=(\tilde{\sigma},\bar{\sigma})\in\tilde{\cU}$, there holds the gradient estimate $Du\bar{\sigma}\in\cL^2((L^2)^{m_1})$. In particular, BSPDE \eqref{SHJB} can be written into the form of BSPDE \eqref{SHJB-SP} and if there exists $\sigma_0=(\tilde\sigma_0,\bar\sigma_0)\in U$ such that $\bar\sigma_0\bar\sigma_0'>0$, then one has further $Du\in \cL^2((L^2)^d)$ and $\psi\in\cL^2((L^2)^{m_0})$.
\end{prop}
To study further the regularity of the solution $(u,\psi)$ for BSPDE \eqref{SHJB}, we assume

\medskip
   $({\mathcal A} 3)$ \it For each $(t,\sigma)\in [0,T]\times U$, $G(\cdot)$ and
$
  f(t,\cdot,\sigma)
$
are deterministic functions on $\bR^d$.\rm

\medskip
Under assumptions $(\cA1)$ and $(\cA3)$, the control problem \eqref{Control-probm} is of Markovian type and has been extensively studied (see \cite{Flem-Soner-2006controlled} and references therein). This can also be seen as a particular case of Proposition \ref{prop-parabolc-BSPDE} with $m_0=0$ and $m_1=m$.

\begin{cor}\label{cor-markovian}
Under assumptions $(\cA1)$ and $(\cA3)$, for the unique weak solution $(u,\psi)$ of BSPDE \eqref{SHJB}, $u$ is deterministic and $\psi\equiv 0$ and for any non-random control $\sigma$, there holds the gradient estimate $Du{\sigma}\in\cL^2((L^2)^{m_1})$. In particular,  BSPDE \eqref{SHJB} is equivalent to the following deterministic PDE
{\small
\begin{equation}\label{SHJB-determn}
  \left\{\begin{array}{l}
  \begin{split}
  -\partial_tu(t,x)=\,&\displaystyle  
  \essinf_{\sigma\in U} \bigg\{\text{tr}\left(\frac{1}{2}{\sigma}{\sigma}'D^2 u\right)(t,x)
        +f(t,x,\sigma)
                \bigg\}, \quad
                     (t,x)\in Q;\\
    u(T,x)=\, &G(x), \quad x\in\bR^d.
    \end{split}
  \end{array}\right.
\end{equation}}
\end{cor}

\begin{rmk}\label{rmk-determn-HJB}
According to the viscosity solution theory for HJB equations (see \cite{crandall-1992-usr-gude,Flem-Soner-2006controlled} for instance), PDE \eqref{SHJB-determn} admits a unique viscosity solution $u$ which coincides with the value function $V(t,x)$ and hence with the weak solution in the sense of Definition \ref{def-weak-sltn-HJB}. In particular, if we assume further that there exists a positive constant $\kappa$ such that for any $\sigma\in U$, $\sigma\sigma'>\kappa \mathbb{I}^{d\times d}$, then PDE \eqref{SHJB-determn} is uniformly parabolic and by the regularity estimates of the viscosity solutions for fully nonlinear parabolic PDE (see \cite{crandall2000lp,krylov-1987,wang1992-I,wang1992-II}), there exists $\bar\alpha\in(0,1)$ such that  $u\in C^{1+\frac{\bar\alpha}{2},2+\bar\alpha}([0,T-\eps]\times\bR^d)$ for any $\eps\in(0,T)$, where the \textit{time-space} H\"older space $C^{1+\frac{\bar\alpha}{2},2+\bar\alpha}([0,T-\eps]\times\bR^d)$ is defined as usual. Then PDE holds in the classical sense and applying the It\^o-Wentzell formula, we have for each $\bar{\sigma}\in\cU$,
{\small
\begin{align*}
&\mu^{\bar{\sigma}}(dt,dx)
\\
&=\left(
                \text{tr}\left(\frac{1}{2}\bar{\sigma} \bar{\sigma}'D^2 u+\bar{\sigma} D\psi^{\bar{\sigma}}\right)(t,x)
        +f(t,x,\bar{\sigma})
             -\essinf_{\sigma\in U} \bigg\{\frac{1}{2}\text{tr}\left({\sigma}{\sigma}'D^2 u\right)(t,x)
        +f(t,x,\sigma)
                \bigg\}
                 \right)
                 dtdx\\
&dK^{\bar{\sigma}}_t(x)
\\
&=
\bigg(
                \text{tr}\left(\frac{1}{2}\bar{\sigma} \bar{\sigma}'D^2 u+\bar{\sigma} D\psi^{\bar{\sigma}}\right)(t,x+X^{\bar{\sigma}}_t)+f(t,x+X^{\bar{\sigma}}_t,\bar{\sigma})
        \\
        &
        \quad\quad\quad
                -\essinf_{\sigma\in U} \bigg\{\frac{1}{2}\text{tr}\left({\sigma}{\sigma}'D^2 u\right)(t,x+X^{\bar{\sigma}}_t)
        +f(t,x+X^{\bar{\sigma}}_t,\sigma)
                \bigg\}  \bigg)
                 dt.
\end{align*}
}
When $\bar{\sigma}$ chosen to be deterministic process, there holds $\psi^{\bar{\sigma}}\equiv 0$ and PDE \eqref{SHJB-determn} also writes
{\small
\begin{equation}\label{SHJB-determn-sigma}
  \left\{\begin{array}{l}
  \begin{split}
  -\partial_tu(t,x)+\mu^{\bar{\sigma}}(dt,x)=\,&\displaystyle  
  \text{tr}\left(\frac{1}{2}{\bar{\sigma}}{\bar{\sigma}}'D^2 u\right)(t,x)
        +f(t,x,\bar{\sigma})
                , \quad
                     (t,x)\in Q;\\
    u(T,x)=\, &G(x), \quad x\in\bR^d.
    \end{split}
  \end{array}\right.
\end{equation}}
\end{rmk}

\section{Comment}

Notice that in the solution pair of Theorem \ref{thm-EU-main}, $u$ inherits the integrability and spacial H\"older continuity of coefficients $f$ and $G$. With the contraction mapping principle, under certain assumptions it seems reasonable to extend the results of above section to the fully nonlinear stochastic HJB equation of the following form:
\begin{equation}\label{SHJB-nonlinear}
  \left\{\begin{array}{l}
  \begin{split}
  -du(t,x)=\,&\displaystyle 
  \essinf_{\sigma\in U} \bigg\{\text{tr}\left(\frac{1}{2}\sigma \sigma' D^2 u+\sigma D\psi\right)(t,x)
        +f(t,x,u,\psi+Du\sigma,\sigma)
                \bigg\} \,dt\\ &\displaystyle
           -\psi(t,x)\, dW_{t}, \quad
                     (t,x)\in Q;\\
    u(T,x)=\, &G(x), \quad x\in\bR^d,
    \end{split}
  \end{array}\right.
\end{equation}
where $f$ is allowed to depend on $u$, $Du$ and $\psi$.  On the other hand, Peng's open problem (see \cite{Peng_92,peng2011backward}) is claimed for BSPDEs of the following form:
{\small
\begin{equation}\label{SHJB-peng-nonlinear}
  \left\{\begin{array}{l}
  \begin{split}
  -du(t,x)=\,&\displaystyle 
  \essinf_{v\in U} \bigg\{\text{tr}\left(\frac{1}{2}\sigma \sigma'(t,x,v) D^2 u(t,x)+\sigma(t,x,v) D\psi(t,x)\right)
       +b'(t,x,v)Du(t,x) \\ &\displaystyle
          \quad\quad\quad+f(t,x,\sigma)
                \bigg\} \,dt -\psi(t,x)\, dW_{t}, \quad
                     (t,x)\in Q;\\
    u(T,x)=\, &G(x), \quad x\in\bR^d,
    \end{split}
  \end{array}\right.
\end{equation}
}
while throughout this work, the controlled leading coefficients of BSPDE \eqref{SHJB} are assumed to be space-homogeneous and $b\equiv 0$ for simplicity. The key reason is that we need the norm equivalence relationship \eqref{eq-equiv-norm}. Hence, to study the general case like BSPDE \eqref{SHJB-peng-nonlinear}, we need to verify the associated  norm equivalence relationships under certain assumptions on $b$ and $\sigma$. We would postpone such generalizations along the  above two lines to future work, as many additional technical efforts are needed.

\begin{appendix}
\section{Generalized It\^o-Wentzell formula by Krylov \cite{Krylov_09}}

Denote  by $\mathscr{D}$ the space of real-valued Schwartz distributions on
$C^{\infty}_{c}(\bR^d)$.  By $\mathfrak{D}$ we denote the set of all $\mathscr{D}$-valued functions defined on
$\Omega\times [0,T]$ such that, for any $u\in \mathfrak{D}$ and $\phi\in C_c^{\infty}$, the
function $\langle u,\,\phi\rangle$ is $\sP$-measurable.

For $p=1,2$ we denote by $\mathfrak{D}^p$ the totality of $u\in\mathfrak{D}$ such that for any $R_1,R_2\in(0,\infty)$ and $\phi\in C_c^{\infty}$, we have
$$
\int_0^{R_2} \sup_{|x|\leq R_1} |\langle u(t,\cdot),\phi(\cdot-x)\rangle|^p \,dt<\infty \quad \text{a.s.}
$$

For $u,f,g\in \mathfrak{D}$, we say that the equality
\begin{equation}\label{eq 1}
du(t,x)=f(t,x)\,dt+g(t,x)\,dW_t, \quad t\in [0,T],
\end{equation}
holds in the sense of distribution if $f {1}_{[0,T]} \in \mathfrak{D}^1$, $g {1}_{[0,T]} \in \mathfrak{D}^2$ and for any $\phi\in C_c^{\infty}$ with probability one we have for all $t\in[0,T]$
$$
\langle u(t,\cdot),\,\phi\rangle=
\langle u(0,\cdot),\,\phi\rangle+\int_0^t\langle f(s,\cdot),\,\phi\rangle\,ds+\int_0^t\langle g(s,\cdot)\,dW_s,\,\phi\rangle.
$$

Let $x_t$ be an $\bR^d$-valued predictable process of the following form
$$
x_t=\int_0^tb_s\,ds+\int_0^t\beta_s\,dW_s,
$$
where $b$ and $\beta$ are predictable processes such that for all $\omega\in\Omega$ and $s\in[0,T]$,
we have
$$
{\rm tr } (\beta_s\beta'_s) <\infty \quad \hbox { \rm and } \quad \int_0^T [|b_t|+{\rm tr } (\beta_t\beta'_t) ]\,dt<\infty.
$$
\begin{thm}[Theorem 1 of \cite{Krylov_09}]\label {Ito-Wentzell}
  Assume that \eqref{eq 1} holds in the sense of distribution and define
  $$
  v(t,x):=u(t,x+x_t).
  $$
  Then we have
  \begin{equation*}
    \begin{split}
      dv(t,x)
      =&\bigg(
      f(t,x+x_t)+\text{tr}\left\{ \frac{1}{2}\beta_t\beta'_tD^2v(t,x)+\beta_t Dg(t,x+x_t)\right\}
      +b'_tDv(t,x)\bigg)\,dt\\
      &+\left( g(t,x+x_t)+Dv(t,x)\beta_t \right)\,dW_t
      ,\quad t\in[0,T]
    \end{split}
  \end{equation*}
  holds in the sense of distribution.
\end{thm}
We note that in the It\^o-Wentzell formula by Krylov \cite[Theorem 1]{Krylov_09}, the Wiener process $(W_t)_{t\geq 0}$ can be general separable Hilbert space-valued.

\section{Banach space-valued BSDEs and degenerate BSPDEs}\label{sec:Banch-BSDE}

For given $p\in(1,\infty)$ and $n\in
(-\infty,\infty)$, we denote by $H ^{n}_{p}$ the space of Bessel
potentials, that is
$$H ^{n}_{p}:=\{\phi\in \mathscr{D}:(1-\Delta)^{\frac{n}{2}}\phi\in L^{p}(\bR ^{d})\}$$
with the Sobolev norm
$$\|\phi\|_{n,p}:=\|(1-\Delta)^{\frac{n}{2}}\phi\|_{L^p(\bR^d)}, ~~ \phi\in
H^{n}_{p}.$$ 
It
is well known that $H ^{n}_{p}$ is a Banach space with the norm
$\|\cdot\|_{n,p}$ and  $C^{\infty}_{c}$ is dense in $H
^{n}_{p}$.

To the end, we take $p\in(1,\infty)$ and $n\in\bR$. Denote by $\bH^{n}_{p,2}$ the subspace of $\mathfrak{D}$ such that,
\[
\|u\|_{\bH_{p,2}^{n}}:=\left\{E\int_{\bR^d}\left(\int_{0}^{T}
    |(1-\Delta)^{\frac{n}{2}}u(t,x)|^{2}dt\right ) ^{\frac{p}{2}}dx\right\}^{1/p}
    <\infty,\quad \forall\, u\in\bH^{n}_{p,2}.
\]

\begin{defn}
    For a function $u\in
    \cL^p(H^n_p),$ we write $u\in\bH_{p,\infty}^n$ if

    (i) there exists $A(u)\in\sF_T\times\cB(\bR^d)$,
    $\bP\otimes dx(A(u))=0$, such that
        for any
        $(\omega,x)\in  \Omega\times \bR^d \setminus A(u),$
        $(1-\Delta)^{n/2}u(\cdot,x)$ is continuous on $[0,T];$

    (ii) $\|u\|_{\bH_{p,\infty}^{n}}:=
        \left(E\int_{\bR^{d}}\sup_{t\in[0,T]}
        |(1-\Delta)^{\frac{n}{2}}u(t,x)|^{p}\ dx\right)^{1/p}<\infty.$
\end{defn}

 Consider functional
\[
F:\Omega\times[0,T]\times\bR^d\times \cL^p(H_p^n)\times (\bH_{p,2}^n)^m \rightarrow \bR
\]
which satisfies: (i) for each $(u,\psi)\in  \cL^p(H_p^n)\times (\bH_{p,2}^n)^m$, $F(\cdot,\cdot,u,\psi)\in\cL^p(H_p^n)$; (ii) there exists $L_2>0$ such that for any $(u_1,\psi_1), (u_2,\psi_2)\in \cL^p(H_p^n)\times (\bH_{p,2}^n)^m$,
\begin{align*}
&\|\left(F(\cdot,\cdot,u_1,\psi_1)-F(\cdot,\cdot,u_2,\psi_2)\right)1_{[t,T]}\|_{\cL^p(H^n_p)}\\
&\leq L_2 \left(\|(u_1-u_2)1_{[t,T]}\|_{\cL^p(H^n_p)}+\sum_{i=1}^m\|(\psi^i_1-\psi^i_2)1_{[t,T]}\|_{\bH^n_{p,2}}\right)
\quad\text{a.s.},\quad \forall\,t\in[0,T].
\end{align*}
Given $\Psi\in L^{p}(\Omega,\sF_{T};H_{p}^{n})$, consider the BSDE
\begin{equation}\label{bsde}
  \left\{\begin{array}{l}
    \begin{split}
      -du(t,x)=F(t,x,u,\psi)\,dt-\psi(t,x)\,dW_t,~~~(t,x)\in Q,
    \end{split}\\
    \begin{split}
      u(T,x)=\Psi(x),~~~~~~~~~x\in \bR^{d}.
    \end{split}
  \end{array}\right.
\end{equation}

\begin{prop}\label{prop-banah-BSDE}
   For $(F,\Psi)$ given above, there exists a unique pair $(u,\psi)\in\left(\cL^p(H_{p}^{n})\cap
  \bH_{p,\infty}^{n}\right)\times\left(\bH_{p,2}^{n}\right)^m$ such that for any $ \phi \in \mathcal{D}_T$
   and $\tau\in[0, T],$ there holds a.s.
\begin{align*}
    \langle u(\tau,\cdot),\phi(\tau) \rangle= \langle \Psi,\phi(T) \rangle+
            \!\int_{\tau}^{T}\!\!\!\left(  \langle F (s,\cdot,u,\psi),\phi(s) \rangle - \langle u,\,\partial_s\phi \rangle (s) \right)\, ds
                    -\!\int_{\tau}^{T}\!\!\! \langle \psi(s)\, dW_{s},\phi(s) \rangle,
\end{align*}
  with
  {\small
      \begin{equation}\label{Estimate of bsde}
      \|u\|_{\bH_{p,\infty}^{n}}+ \|u\|_{\cL^p(H_p^n)} +\sum_{i=1}^m\|\psi^i\|_{\bH_{p,2}^{n}}\leq
      C(p,T,L_2)\left(\|F(\cdot,\cdot,0,0)\|_{\cL^p(H^{n}_{p})}+\|\Psi\|_{L^{p}(\Omega,\sF_{T};H_{p}^{n})}
      \right).
   \end{equation}
   }
\end{prop}
In fact, the above proposition is generalized from  \cite[Lemmas 3.1 and 3.2]{DuQiuTang10} where $F$ is independent of $u$ and $\psi$. Since the proof is just a standard application of Picard iteration for the BSDE theory (see \cite{Karoui_Peng_Quenez,ParPeng_90}), we omit it.

An immediate consequence of Theorem \ref{Ito-Wentzell} and Proposition \ref{prop-banah-BSDE} is the following
\begin{cor}\label{cor-degenerate-BSPDE}
Assume the same hypothesis of Proposition \ref{prop-banah-BSDE} with $p=2$. Given $\sigma\in\cU$, the following BSPDE
\begin{equation*}
  \left\{\begin{array}{l}
  \begin{split}
  -d{u}(t,x)=\,&\displaystyle 
  \left[\text{tr}\left(\frac{1}{2}\sigma \sigma' D^2 {u}+\sigma D{\psi}\right)(t,x)+F(t,x,u,\psi+Du\sigma)\right]
                 \,dt
           -{\psi}(t,x)\, dW_{t},
                     \,\,(t,x)\in Q;\\
    {u}(T,x)=\, &\Psi(x), \quad x\in\bR^d,
    \end{split}
  \end{array}\right.
\end{equation*}
admits a unique solution $(u,\psi)\in \cL^2(H_{2}^{n})\times\cL^2\left((H_{2}^{n-1})^m\right)$, i.e., for any $(t,\varphi)\in[0,T]\times\mathcal{D}_T$,
\begin{equation*}
  \begin{split}
    & \langle  u,\,\varphi  \rangle(t) +\!\!\int_t^T\! \langle u,\,\partial_s \varphi    \rangle(s)  \,ds+\int_t^T \langle \varphi(s),\,\psi(s)\,dW_s \rangle\\
    =\,&
    \langle \Psi,\varphi(T) \rangle+\!\!\int_t^T\!\left[  \langle F(s,\cdot,u,\psi+Du \sigma ),\,\varphi(s)  \rangle
    +\frac{1}{2} \langle u,\, \text{tr}\left(\sigma\sigma'D^2\varphi\right)  \rangle(s) -
     \langle \sigma\psi,\, D\varphi  \rangle (s) \right]\,ds
     ,\,\text{a.s.}
  \end{split}
\end{equation*}
Moreover, for this solution, one has
{\small
\begin{align*}
      &\left(E\int_{\bR^d} \sup_{t\in[0,T]} \left|(1-\Delta)^{\frac{n}{2}}u(t,x+X_t^{\sigma})\right|^2  \,dx\right)^{1/2}+ \|u\|_{\cL^2(H_2^n)} +\|\psi+Du\sigma\|_{\cL^2((H_{2}^{n})^m)}\\
      &\leq
      C\left(\|F(\cdot,\cdot,0,0)\|_{\cL^2(H^{n}_{2})}+\|\Psi\|_{L^{2}(\Omega,\sF_{T};H_{2}^{n})}
      \right),
   \end{align*}
   }
 with $C$ depending on $T$ and $L_2$.
\end{cor}

\begin{rmk}
Corollary \ref{cor-degenerate-BSPDE} seems to be an interesting complement  for the existing literature on  degenerate BSPDEs (for instance, see \cite{DuTangZhang-2013,Horst-Qiu-Zhang-14,Hu_Ma_Yong02,ma1997adapted,ma1999linear}). Indeed, $n$ herein is allowed to be any real number and for the solution, $(1-\Delta)^{n/2}u(t,x+X_t^{\sigma})$ is $\bP\otimes dx$-a.e. time-continuous, and especially, when $n=0$, the associated estimate on $u(t,x+X_t^{\sigma})$ leads to the $\sigma$-quasi-continuity of $u$ in Proposition \ref{prop-quasi-cont-BSPDE}.
\end{rmk}

\end{appendix}

\bibliographystyle{siam}

\end{document}